\documentclass{amsart}

\title[Matrices with normal defect one]{Matrices with normal defect one}

\author[D.~S.~Kaliuzhnyi-Verbovetskyi]{Dmitry S.
Kaliuzhnyi-Verbovetskyi}
\address{Department of Mathematics \\
Drexel University\\
3141 Chestnut Str.\\
 Philadelphia, PA, 19104}
\email{dmitryk@math.drexel.edu}
\author[I. M. Spitkovsky]{Ilya M. Spitkovsky}
\address{Department of Mathematics \\
College of William and Mary\\
 Williamsburg, VA, 23187-8795}
\email{ilya@math.wm.edu}
\author[H. J. Woerdeman]{Hugo J. Woerdeman}
\address{Department of Mathematics \\
Drexel University\\
3141 Chestnut Str.\\
 Philadelphia, PA, 19104}
\email{hugo@math.drexel.edu}
\thanks{I. M. Spitkovsky and H. J. Woerdeman are partially supported by NSF grants DMS-0456625 and
DMS-0500678, respectively.}

\date{}
\newtheorem{thm}{Theorem}[section]

\newtheorem{cor}[thm]{Corollary}
\newtheorem{lem}[thm]{Lemma}
\newtheorem{rem}[thm]{Remark}

\newtheorem{ex}[thm]{Example}

\renewcommand{\Re}{\operatorname{Re}}
\renewcommand{\Im}{\operatorname{Im}}
\newcommand{\nd}{\operatorname{nd}}
\newcommand{\rnd}{\operatorname{rnd}}
\newcommand{\chd}{\operatorname{chd}}
\newcommand{\csd}{\operatorname{csd}}

\newcommand{\Null}{\operatorname{null}}
\newcommand{\ud}{\operatorname{ud}}
\newcommand{\diag}{\operatorname{diag}}
\newcommand{\col}{\operatorname{col}}

\newcommand{\rank}{\operatorname{rank}}
\newcommand{\spn}{\operatorname{span}}
\newcommand{\trace}{\operatorname{trace}}
\newcommand{\range}{\operatorname{range}}
\newcommand{\norm}[1]{\left\Vert#1\right\Vert}
\numberwithin{equation}{section}



\begin{document}

\maketitle

\begin{abstract}
A $n\times n$ matrix $A$ has normal defect one if it is not normal,
however can be embedded as a north-western block into a normal
matrix of size $(n+1)\times (n+1)$. The latter is called  a minimal
normal completion of $A$. A construction of all matrices with normal
defect one is given. Also, a simple procedure is presented which
allows one to check whether a given matrix has normal defect one,
and if this is the case --- to construct all its minimal normal
completions. A characterization of the generic case for each $n$
under the assumption  $\rank(A^*A-AA^*)=2$ (which is necessary for
$A$ to have normal defect one) is obtained. Both the complex and the
real cases are considered. It is pointed out how these results can be
used to solve the minimal commuting completion problem in the
classes of pairs of $n\times n$ Hermitian (resp., symmetric, or
symmetric/antisymmetric) matrices when the completed matrices are
sought of size $(n+1)\times(n+1)$. An application to the $2\times n$
separability problem in quantum computing is described.
\end{abstract}

\section{Introduction}\label{sec:intro}
A matrix $N\in\mathbb{C}^{n\times n}$ is called \emph{normal} if
$N^*N=NN^*$. For a non-normal
$A\in\mathbb{C}^{n\times n}$ it is natural to inquire what is the
minimal $p\in\mathbb{N}$ such that $A$ has a \emph{normal
completion} of size $(n+p)\times(n+p)$. In other words, what is
the smallest $p$ for which there exists a matrix of the form
\begin{equation}\label{comp} \begin{bmatrix} A & *\\
* & * \end{bmatrix}\in\mathbb{C}^{(n+p)\times (n+p)}\end{equation}
which is in fact normal? This minimal $p$ is called the \emph{normal
defect of $A$}, denoted $\nd(A)$, and a normal completion of size
$(n+\nd(A))\times(n+\nd(A))$ is called \emph{minimal}.

The normal completion problem as above was introduced in \cite{W},
and some observations were made there. One of them is that among
completions (\ref{comp}) there exist those being scalar multiples of a
unitary matrix. The smallest value of $p$ required for such a
completion is called the \emph{unitary
defect of $A$}, denoted $\ud(A)$, and the corresponding completions
are called \emph{minimal unitary completions of $A$}. Obviously,
$\nd(A)\le\ud(A)$.

The latter inequality gives a simple upper bound for  $\nd(A)$. Indeed,
$\ud(A)$ is simply  the number (counting the multiplicities) of the
singular values of $A$ different from $\norm{A}$, and is therefore
strictly less than $n$. Moreover, it was shown in \cite{W} that a
minimal unitary completion of $A$ can be constructed using the
singular value decomposition (SVD) of $A$ (see, e.g., \cite{HJ} for the
definition of the SVD).

It is easy to find examples of matrices with $\nd(A)$ different
from $\ud(A)$. For instance, if $A$ is normal and not a multiple
of a unitary matrix then $\nd(A)=0<\ud(A)$. However, in all such
examples known until recently, the matrix $A$ was {\em unitarily
reducible}, that is, unitarily similar to a block diagonal matrix
with non-trivial blocks of smaller size. It was therefore natural
to ask  \cite{W} (see also a further discussion in \cite{KW})
whether  the equality $\nd(A)=\ud(A)$ holds for all unitarily
irreducible matrices $A\in\mathbb{C}^{n\times n}$. We will show in
Examples \ref{ex:a} and \ref{ex:b} that this question has a
negative answer, and thus normal defect is indeed a different
characteristic of a matrix then its unitary defect.

A lower bound for $\nd(A)$ has been found in \cite{KW}:
\begin{equation}\label{lower_bound}
 \nd (A)\ge\max\{i_+(A^*A-AA^*),i_-(A^*A-AA^*)\},
\end{equation}
where $i_+(X)$ and $i_-(X)$ denote the number of positive
(negative) eigenvalues of a Hermitian matrix $X$.

For $2\times 2$ matrices, the unitary defect is at most $1$.
Therefore, any non-normal matrix $A$ of size $2\times 2$ has
normal defect $1$. Since $\trace(A^*A-AA^*)=0$, the  righthand
side of (\ref{lower_bound}) is 0 if $A\in{\mathbb C}^{2\times 2}$
is normal and 1 otherwise. In other words, for $n=2$
(\ref{lower_bound}) turns into an equality.  We will show in
Corollary \ref{n=3} that this is also true for matrices of size
$3\times 3$. On the other hand, our Example \ref{rc_ns} reveals
that for larger matrices a strict inequality in
\eqref{lower_bound} is possible.

The bulk of this paper is devoted to matrices with normal defect one.
According to (\ref{lower_bound}), all such matrices $A$ must satisfy
\begin{equation}\rank(A^*A-AA^*)=2\label{rank} \end{equation}
which throughout the paper will be referred to as the {\em rank
condition}. The manifold of $n\times n$ matrices satisfying
(\ref{rank}) will be denoted ${\frak M}_n$. \label{manif}

We obtain several equivalent characterizations of matrices
$A\in\mathbb{C}^{n\times n}$ with $\nd(A)=1$. Among others, we
prove in Section~\ref{sec:compl} that the following statements are
equivalent: \emph{\begin{itemize}
    \item[(i)] $\nd(A)=1$.
    \item[(ii)] There exist a contraction matrix $C\in\mathbb{C}^{n\times n}$ with $\ud(C)=1$, a diagonal
    matrix $D\in\mathbb{C}^{n\times n}$, and a scalar $\mu\in\mathbb{C}$
    such that
    \begin{equation*}
    A=CDC^*+\mu I_n.
    \end{equation*}
    \item[(iii)] The rank condition holds, and
    the equation
    \begin{equation*}
PA^*(x_1u_1+x_2u_2)=PA(\overline{x}_2u_1+\overline{x}_1u_2)
\end{equation*}
has a solution pair $x_1,x_2\in\mathbb{C}$ satisfying
\begin{equation*}
|x_1|^2-|x_2|^2=d.
\end{equation*}\end{itemize}
Here $u_1,u_2$ are the unit eigenvectors of the matrix $A^*A-AA^*$
corresponding to its nonzero eigenvalues $\lambda_1=d (>0)$ and
$\lambda_2=-d$, and $P=I_n-u_1u_1^*-u_2u_2^*$
 is the orthogonal projection of
$\mathbb{C}^n$ onto $\Null(A^*A-AA^*)$. \begin{itemize}    \item[(iv)]
There exist linearly independent $x, y \in {\mathbb C}^n$ such that
\begin{equation*} A^*A-AA^* = xx^* - yy^* \end{equation*}
and the vectors $x,y, A^*x, Ay$ are linearly dependent.
\end{itemize}
} The equivalences (i)$\Longleftrightarrow$(ii),
(i)$\Longleftrightarrow$(iii), and (i)$\Longleftrightarrow$(iv) are
proved in Theorems \ref{nd1}, \ref{thm1}, and \ref{normalext1},
respectively. Note that version (ii) is better suited for construction of
all matrices $A$ with $\nd(A)=1$ (Section~\ref{construct}), while (iii) is
used  to describe a procedure which allows one to check whether
$\nd(A)=1$, and if this is the case --- to construct all minimal normal
completions of $A$ (Section~\ref{shorter}). Finally, (iv) becomes
handy when solving a separability problem originated in quantum
computing (see description of Section~\ref{sec:sep} below).

Section~\ref{finer} provides a finer analysis which involves a certain
basis construction and a lemma on symmetric and unitary extensions
of certain matrices. This analysis allows us to refine the procedure
from Section \ref{shorter} and to describe the generic situation for
each $n$ under the assumption that rank condition (\ref{rank}) holds.
In other words,
certain topological characterization of the set of matrices with normal
defect one in each matrix dimension is presented.

We also obtain the following result:

\emph{If $A\in\mathbb{R}^{n\times n}$ and $\nd(A)=1$ then one can
choose a minimal normal completion of $A$ to be real as well.}

It is natural to study the minimal normal completion problem also
in the setting of real matrices,  and we treat it in a separate
Section~\ref{sec:re}. The real counterpart of the normal defect of
a matrix $A\in\mathbb{R}^{n\times n}$, denoted $\rnd(A)$, is
defined. The result stated above can be reformulated as follows:
$\rnd(A)=1$ if and only if $\nd(A)=1$. (The question on whether
$\rnd(A)=\nd(A)$ for an arbitrary $A\in\mathbb{R}^{n\times n}$
remains open.) This and other results in the real case are not
immediate consequences of their complex counterparts, and required
an additional study. Some statements in the real case are similar
to their complex analogues, however there are also some
differences. The real counterpart of characterization (ii) of
matrices with normal defect one (see above) is obtained for
matrices $A\in\mathbb{R}^{n\times n}$ of even size $n$ only, while
such a characterization and a construction of all real matrices
with $\rnd(A)=1$ in the case of odd $n$ are left as an open
problem. The real analogue of characterization (iii), as well as
the procedure for verification that $\rnd(A)=1$ and for
construction of all minimal real normal completions, have a
slightly different form which splits into two cases. The generic
situation in each matrix dimension is also described, however in
the real case the analysis happens to be more straightforward than
its counterpart in the complex case.

In Section \ref{sec:commut}, we show how to restate our results
from Sections \ref{sec:compl} and \ref{sec:re} in terms of
commuting completions of a pair of Hermitian (resp., symmetric and
antisymmetric) matrices, where the completed matrices are also
Hermitian (resp., symmetric and antisymmetric). The results for
pairs of Hermitian matrices are used then to solve an analogous
problem in the class of pairs of symmetric matrices.

 In Section \ref{sec:sep}, we use the
connection between the normal completion problem and the
$2\times n$ separability problem, that was established in \cite{Wsep},
to obtain Theorem \ref{sep} which gives the easily verifiable
necessary and sufficient conditions for a positive semidefinite matrix
$M\in\mathbb{C}^{2n\times 2n}$ with a rank one Schur complement
to be $2\times n$ separable. Moreover, a new proof is given for the
result by Woronowicz \cite{Woron} (see Theorem \ref{Woron}), which
establishes, for $n\le 3$,  the $2\times n$ separability for a positive
semidefinite matrix $M\in\mathbb{C}^{2n\times 2n}$ satisfying the
Peres test. 

\section{The complex case}\label{sec:compl}

\subsection{Construction of matrices with normal defect
one.}\label{construct}
\begin{thm}\label{nd1}
Let $A\in\mathbb{C}^{n\times n}$ be not normal. The following
statements are equivalent:
\begin{itemize}
    \item[(i)] $\nd(A)= 1$.
    \item[(ii)] There exist a contraction matrix $C\in\mathbb{C}^{n\times n}$ with $\ud(C)=1$, a diagonal
    matrix $D\in\mathbb{C}^{n\times n}$, and a scalar $\mu\in\mathbb{C}$
    such that
    \begin{equation}\label{repr1}
    A=CDC^*+\mu I_n.
    \end{equation}
    \item[(iii)] There exist a unitary matrix $V\in\mathbb{C}^{n\times n}$,
    a normal matrix $N\in\mathbb{C}^{n\times n}$, and scalars
    $t\colon 0\le t<1$, $\mu\in\mathbb{C}$ such that
    \begin{equation}\label{repr2}
    V^*AV=MNM+\mu I_n,
    \end{equation}
    where $M=\diag(1,\ldots,1,t)$.
\end{itemize}
\end{thm}
\begin{proof}
(i)$\Longleftrightarrow$(ii) Let $\nd(A)= 1$, and let
$\begin{bmatrix} A & x\\
y^* & z
\end{bmatrix}\in\mathbb{C}^{(n+1)\times(n+1)}$ be a minimal normal
completion of $A$. Then there exist a diagonal matrix
$\Lambda\in\mathbb{C}^{n\times n}$, a scalar $\mu\in\mathbb{C}$,
and a unitary matrix $U=\begin{bmatrix} U_{11} & U_{12}\\
U_{21} & U_{22}
\end{bmatrix}\in\mathbb{C}^{(n+1)\times(n+1)}$ such that
\begin{equation}\label{diag}
\begin{bmatrix} A & x\\
y^* & z
\end{bmatrix}=\begin{bmatrix} U_{11} & U_{12}\\
U_{21} & U_{22}
\end{bmatrix}\begin{bmatrix} \Lambda & 0\\
0 & \mu
\end{bmatrix}\begin{bmatrix} U_{11}^* & U_{21}^*\\
U_{12}^* & \overline{U}_{22}
\end{bmatrix}.
\end{equation}
The latter equality is equivalent to the following system:
\begin{eqnarray}
A &= & U_{11}\Lambda U_{11}^*+\mu
U_{12}U_{12}^*=U_{11}(\Lambda-\mu
I_n) U_{11}^*+\mu I_n, \label{A}\\
x &=& U_{11}\Lambda U_{21}^*+\mu
U_{12}\overline{U}_{22}=U_{11}(\Lambda-\mu
I_n) U_{21}^*, \label{x}\\
y^* &=& U_{21}\Lambda U_{11}^*+\mu
{U}_{22}U_{12}^*=U_{21}(\Lambda-\mu
I_n) U_{11}^*, \label{y*}\\
z &= & U_{21}\Lambda U_{21}^*+\mu
U_{22}\overline{U}_{22}=U_{21}(\Lambda-\mu I_n)
U_{21}^*+\mu.\label{z}
\end{eqnarray}
Setting $C=U_{11}$ and $D=\Lambda-\mu I_n$, we obtain
\eqref{repr1} from \eqref{A}.

Conversely, if \eqref{repr1} holds, we set $U_{11}=C$,
$\Lambda=D+\mu I_n$ and obtain \eqref{A}. For $U=\begin{bmatrix} U_{11} & U_{12}\\
U_{21} & U_{22}
\end{bmatrix}$ a minimal unitary completion of $C$, we define
$x,y\in\mathbb{C}^n$ and $z\in\mathbb{C}$ by \eqref{x}--\eqref{z}.
Then \eqref{diag} holds, i.e., the
matrix $\begin{bmatrix} A & x\\
y^* & z
\end{bmatrix}\in\mathbb{C}^{(n+1)\times(n+1)}$ is a normal
completion of $A$, and thus $\nd{A}= 1$.

(ii)$\Longleftrightarrow$(iii) If (ii) holds, let $C=V\diag(1,\ldots,1,t)W^*$
be the SVD of $C$ (here $V,W\in\mathbb{C}^{n\times n}$ are unitary,
$0\le t<1$, and $M=\diag(1,\ldots,1,t)\in\mathbb{C}^{n\times n}$).
Then, clearly, $N=W^*DW$ is normal, and \eqref{repr2} follows.

Conversely, if \eqref{repr2} holds, then $N=W^*DW$ with $D$
diagonal and $W$ unitary. Clearly, for $C=V\diag(1,\ldots,1,t)W^*$
we have $\ud(C)=1$, and \eqref{repr1} follows.
\end{proof}

\begin{rem}\label{not_normal}
\rm{Observe that the matrix $A$ given by  (\ref{repr2}) happens to
be normal if and only if the product $MNM$ is normal, that is
\begin{equation}\label{nocon} MNM^2N^*M=MN^*M^2NM. \end{equation}  Since $N$ itself is normal,
(\ref{nocon}) holds if and only if \begin{equation}\label{nocon1}
MNZN^*M= MN^*ZNM,
\end{equation} where $Z=\diag(0,\ldots,0,1)$. Partitioning $N$ as \[
N=\left[\begin{matrix}N_0 & g \\ h^* & \alpha\end{matrix}\right],
\] where $\alpha$ is scalar, and rewriting (\ref{nocon1})
block-wise, we see that it is equivalent to \[ gg^*=hh^*,\quad
t\alpha h=t\overline{\alpha}g.\] These
conditions mean simply that $g$ differs from $h$  by a scalar
multiple of absolute value one and, if $t\alpha\neq 0$, this
scalar must be $\alpha/\overline{\alpha}$. Consequently, $A$ is
not normal if and only if this is not the case.

Observe also that if $t\neq 0$ (so that $M$ is invertible) and $N$ is
also invertible, then (\ref{nocon}) can be written as
\begin{equation}\label{com} M^2N^*N^{-1}=N^{-1}N^*M^2.\end{equation}  But $N$ is
normal, so that $N^*$ commutes with $N^{-1}$. Condition
(\ref{com}) therefore means simply that $N^*N^{-1}$ ($=N^{-1}N^*$)
commutes with $M^2$. In other words, $A$ in this case is normal if
and only if $e_n:=\col(0,\ldots,0,1)$ is an eigenvector of
$N^*N^{-1}$.  In turn, this happens if and only if $e_n$ belongs
to the sum of eigenspaces of $N$ with the corresponding
eigenvalues lying on the same line through the origin.}
\end{rem}

Representation \eqref{repr1} or \eqref{repr2} in Theorem \ref{nd1},
together with Remark \ref{not_normal}, allow one to construct all
matrices $A$ with $\nd(A)= 1$. However, as we mentioned in Section
\ref{sec:intro}, this does not give an easy way to check whether a
given matrix has normal defect one. A procedure for this is our further
goal.

\subsection{Identification of matrices with nd$\mathbf{(A)=1}$ and construction of all minimal
normal completions of $\mathbf{A}$}\label{shorter} In the following
two theorems, we establish necessary and sufficient conditions for a
matrix $A$ to have normal defect one, and for any matrix $A$ with
$\nd(A)=1$ we describe all its minimal normal completions. Here and
throughout the rest of the paper, we set $\mathbb{T}=\{
z\in\mathbb{C}\colon |z|=1\}$.

\begin{thm}\label{thm1}
Let $A\in\mathbb{C}^{n\times n}$. Then
\begin{itemize}
    \item[(i)] $\nd(A)=1$ if and only if $\rank(A^*A-AA^*)=2$ and
    the equation
    \begin{equation}\label{eq}
PA^*(x_1u_1+x_2u_2)=PA(\overline{x}_2u_1+\overline{x}_1u_2)
\end{equation}
has a solution pair $x_1,x_2\in\mathbb{C}$ satisfying
\begin{equation}\label{cond}
|x_1|^2-|x_2|^2=d.
\end{equation}\end{itemize}
Here $u_1,u_2\in\mathbb{C}^n$ are the unit eigenvectors of the
matrix $A^*A-AA^*$ corresponding to its nonzero eigenvalues
$\lambda_1=d (>0)$ and $\lambda_2=-d$, and
\begin{equation}\label{P}
P=I_n-u_1u_1^*-u_2u_2^*
\end{equation}
 is the orthogonal projection of
$\mathbb{C}^n$ onto $\Null(A^*A-AA^*)$. \begin{itemize}    \item[(ii)]
If $\nd(A)=1$, $x_1$ and $x_2$ satisfy \eqref{eq} and \eqref{cond}, and
$\mu\in\mathbb{T}$ is arbitrary then the matrix
\begin{equation}\label{ncompl}
B=\begin{bmatrix} A & \mu(x_1u_1+x_2u_2)\\
\overline{\mu}({x}_2u_1^*+{x}_1u_2^*) & z
\end{bmatrix}
\end{equation}
is a minimal normal completion of $A$. Here
\begin{equation}\label{z_def}
z=a_{11}-\frac{1}{d}\left(x_2(a_{12}\overline{x}_1-\overline{a}_{21}x_2)+
x_1(\overline{a}_{12}x_1-a_{21}\overline{x}_2)\right)
\end{equation}
and
\begin{equation}\label{a's}
a_{11}=u_1^*Au_1,\quad a_{12}=u_1^*Au_2,\quad a_{21}=u_2^*Au_1.
\end{equation}\end{itemize}
All minimal normal completions of $A$ arise in this way.
\end{thm}
\begin{thm}\label{normalext1} Let $A\in\mathbb{C}^{n\times n}$. Then $\nd (A)=1$ if and only if
there exist linearly independent $x, y \in {\mathbb C}^n$ such
that
\begin{equation}\label{xy_eq} A^*A-AA^* = xx^* - yy^* \end{equation}
and the vectors $x,y, A^*x, Ay$ are linearly dependent. In this
case, there exist $z\in {\mathbb C}$ and $\nu\in\mathbb{T}$ such
that the matrix
\begin{equation}\label{nm}
B=\begin{bmatrix} A & \nu x \\ y^* & z
\end{bmatrix}\end{equation} is normal.
\end{thm}

In order to prove Theorems \ref{thm1} and \ref{normalext1} we will
need several auxiliary statements.

\begin{lem}\label{nec_suf_cond}
Let $A\in\mathbb{C}^{n\times n}$. Then $\nd(A)=1$ if and only if
there exist linearly independent vectors $x,y\in\mathbb{C}^n$ and
a scalar $z\in\mathbb{C}$ such that
\begin{equation}\label{id1}
A^*A-AA^*=xx^*-yy^*,
\end{equation}
\begin{equation}\label{id2}
(A-zI_n)^*x=(A-zI)y.
\end{equation}
\end{lem}
\begin{proof}
 If $\nd(A)=1$ then there exists a normal matrix $B=\begin{bmatrix} A & x\\
y^* & z
\end{bmatrix}\in\mathbb{C}^{(n+1)\times(n+1)}$. The identity $B^*B=BB^*$ is equivalent to
\eqref{id1}--\eqref{id2} (the identity $x^*x=y^*y$ follows from
\eqref{id1} since $\trace(A^*A-AA^*)=0$, and is therefore
redundant). Clearly, $x$ and $y$ are linearly independent,
otherwise the right-hand side of \eqref{id1} is $0$ and $A$ is
normal.

Conversely, if $x,y\in\mathbb{C}^n$ are linearly independent,
$z\in\mathbb{C}$, and \eqref{id1}--\eqref{id2} hold then the
matrix $B=\begin{bmatrix} A & x\\
y^* & z
\end{bmatrix}\in\mathbb{C}^{(n+1)\times(n+1)}$ is normal. Since
the right-hand side of \eqref{id1} is not $0$, the matrix $A$ is
not normal, thus $\nd(A)=1$.
\end{proof}
\begin{cor}\label{nec_cond}
If $\nd(A)=1$ then $\rank(A^*A-AA^*)=2$.
\end{cor}
The rank condition above is easy to check. Its equivalent form
 is that the characteristic polynomial of $A^*A-AA^*$ can be
 written as
\begin{equation}\label{rank_cond}
\det( A^*A-AA^*-\lambda I_n)=(-1)^n\lambda^{n-2}(\lambda^2-d^2),
\end{equation}
with some $d>0$. If this condition is satisfied, one can find the
unit eigenvectors $u_1$ and $u_2$ of the matrix $A^*A-AA^*$
corresponding to its eigenvalues $\lambda_1=d$ and $\lambda_2=-d$,
which are determined uniquely up to a scalar factor. There is more
freedom in a choice of other eigenvectors $u_3$, \ldots, $u_n$,
which form an orthonormal basis of $\Null(A^*A-AA^*)$. Suppose
that such vectors are chosen. Then $U=\begin{bmatrix} u_1 & u_2 &
u_3 & \ldots & u_n \end{bmatrix}\in\mathbb{C}^{n\times n}$ is a
unitary matrix, and the matrix $\widetilde{A}=U^*AU$ satisfies
\begin{equation}\label{id1'}
\widetilde{A}^*\widetilde{A}-\widetilde{A}\widetilde{A}^*=H,
\end{equation}
where
\begin{equation}\label{h}
H=\diag(d, -d, 0,\ldots, 0)\in\mathbb{C}^{n\times n}.
\end{equation}
\begin{lem}\label{other_nec}
If $\widetilde{A}\in\mathbb{C}^{n\times n}$ satisfies \eqref{id1'}
then $\widetilde{A}$ has the form
\begin{equation}\label{A_form}
\widetilde{A}=\begin{bmatrix} a_{11} & a_{12} & u\\
a_{21} & a_{22} & v\\
w^* & q^* & S
\end{bmatrix},
\end{equation}
where $a_{ij}$ ($i, j=1,2$) are scalars,
\begin{equation}\label{adiag} a_{11}=a_{22},
\end{equation}
and $u^*,v^*,w^*,q^*\in\mathbb{C}^{n-2}$ satisfy
\begin{equation}\label{uvqw}
uu^*=qq^*,\quad vv^*=ww^*,\quad uv^*=wq^*,\quad uw^*=vq^*.
\end{equation}
\end{lem}
\begin{proof}
We have
$$\trace (\widetilde{A}H)=\trace(\widetilde{A}\widetilde{A}^*\widetilde{A}-\widetilde{A}^2\widetilde{A}^*)=0,$$
 which implies
\eqref{adiag}. Then, from the equality
$$(\widetilde{A}^*\widetilde{A}-\widetilde{A}\widetilde{A}^*)_{12}=0$$
we obtain that $uv^*=wq^*$. Next, from the observation
\begin{eqnarray*}
\trace(\widetilde{A}^*H\widetilde{A}+\widetilde{A}H\widetilde{A}^*)
&=&
\trace(\widetilde{A}^*(\widetilde{A}^*\widetilde{A}-\widetilde{A}\widetilde{A}^*)\widetilde{A}+\widetilde{A}
(\widetilde{A}^*\widetilde{A}-\widetilde{A}\widetilde{A}^*)\widetilde{A}^*)\\
&=&
\trace(\widetilde{A}^{*2}\widetilde{A}^2-\widetilde{A}^*\widetilde{A}\widetilde{A}^*\widetilde{A}+
\widetilde{A}\widetilde{A}^*\widetilde{A}\widetilde{A}^*-\widetilde{A}^2\widetilde{A}^{*2})=0
\end{eqnarray*} we obtain
\begin{equation}\label{+-+-}
uu^*-vv^*+ww^*-qq^*=0,
\end{equation}
and from the equality
$$\trace((\widetilde{A}^*\widetilde{A}-\widetilde{A}\widetilde{A}^*)_{33})=0$$
we obtain
\begin{equation}\label{++--}
uu^*+vv^*=ww^*+qq^*.
\end{equation}
Combining \eqref{+-+-} and \eqref{++--}, we obtain that
$uu^*=qq^*$ and $vv^*=ww^*$. From the observation
\begin{eqnarray*}
\trace(\widetilde{A}H\widetilde{A}) &=& \trace(\widetilde{A}(\widetilde{A}^*\widetilde{A}-
\widetilde{A}\widetilde{A}^*)\widetilde{A}) \\
&=&
\trace(\widetilde{A}\widetilde{A}^*\widetilde{A}^2-\widetilde{A}^2\widetilde{A}^*\widetilde{A})=0
\end{eqnarray*}
we obtain that $uw^*=vq^*$.
\end{proof}
\begin{lem}\label{xy}
Suppose that
\begin{equation}\label{A_form'}
\widetilde{A}=\begin{bmatrix} a_{11} & a_{12} & u\\
a_{21} & a_{11} & v\\
w^* & q^* & S
\end{bmatrix}\in\mathbb{C}^{n\times n}
\end{equation}
 satisfies \eqref{id1'}
with $H$ given by \eqref{h}. Then the matrix
$\widetilde{B}=\begin{bmatrix} \widetilde{A} & x\\
y^* & z\end{bmatrix}\in\mathbb{C}^{(n+1)\times (n+1)}$ is normal
if and only if
\begin{equation}\label{xy_new}
x=\col(x_1,x_2,0,\ldots,0)\in\mathbb{C}^n,\quad
y=\col(y_1,y_2,0,\ldots,0)\in\mathbb{C}^n,
\end{equation}
\begin{equation}\label{diff}
|x_1|^2-|x_2|^2=d,
\end{equation}
\begin{equation}\label{y_via_x}
y_1=e^{i\theta}\overline{x}_2,\quad y_2=e^{i\theta}\overline{x}_1,
\end{equation} for some $\theta\in\mathbb{R}$,
and the following identities hold:
\begin{eqnarray}
(\overline{a}_{11}-\overline{z})x_1+\overline{a}_{21}x_2 &=& (a_{11}-z)y_1+a_{12}y_2,\label{eq1}\\
\overline{a}_{12}x_1+(\overline{a}_{11}-\overline{z})x_2 &=&
a_{21}y_1+(a_{11}-z)y_2,\label{eq2}\\
 u^*x_1+v^*x_2 &=&
w^*y_1+q^*y_2.\label{eq3}
\end{eqnarray}
\end{lem}
\begin{proof}
 It follows from Lemma \ref{nec_suf_cond} applied to the
 matrix $\widetilde{A}$
as above (see also Lemma \ref{other_nec} which justifies
that $a_{22}=a_{11}$) that $\nd(\widetilde{A})=1$ and $\widetilde{B}=\begin{bmatrix} \widetilde{A} & x\\
y^* & z\end{bmatrix}\in\mathbb{C}^{(n+1)\times (n+1)}$ is a
minimal normal completion of $\widetilde{A}$ if and only if $x$
and $y$ are linearly independent and \eqref{id1}--\eqref{id2} hold
with $A$ replaced by $\widetilde{A}$. Since in this case
$H=xx^*-yy^*$, the vectors $x$ and $y$ have the form
\eqref{xy_new}. Indeed, for any vector $h\in\mathbb{C}^n$ which is
orthogonal to $y$ and not orthogonal to $x$, we have
$$0\neq Hh=xx^*h\in\range(H)\cap \spn(x).$$
Similarly, for any vector $g\in\mathbb{C}^n$ which is orthogonal to
$x$ and not orthogonal to $y$, we have
$$0\neq Hg=-yy^*g\in\range(H)\cap \spn(y),$$
thus both $x$ and $y$ are in $\range(H)$. Next, the identity
$H=xx^*-yy^*$ holds if and only if
$$|x_1|^2-|x_2|^2=d=|y_2|^2-|y_1|^2,\quad
x_1\overline{x}_2=y_1\overline{y}_2,$$  or equivalently,
\eqref{y_via_x} holds with some $\theta\in\mathbb{R}$. Clearly,
\eqref{id2} with $A$ replaced by $\widetilde{A}$, is equivalent to
\eqref{eq1}--\eqref{eq3}.
\end{proof}
\begin{rem}\label{x_adj}
\rm{If $\widetilde{A}$ is as in Lemma \ref{xy} and $\widetilde{B}=\begin{bmatrix} \widetilde{A} & x\\
y^* & z\end{bmatrix}\in\mathbb{C}^{(n+1)\times (n+1)}$ is a
minimal normal completion of $\widetilde{A}$ then so is
\begin{equation*}
\begin{bmatrix} \widetilde{A} & \mu x\\
\overline{\mu}y^* & z\end{bmatrix}=\begin{bmatrix} I_n & 0\\
0 & \overline{\mu}
\end{bmatrix}\begin{bmatrix} \widetilde{A} & x\\
y^* & z\end{bmatrix}\begin{bmatrix} I_n & 0\\
0 & \mu
\end{bmatrix}
\end{equation*}
for any $\mu\in\mathbb{T}$. Therefore, if $x$, $y$ and $z$ are as
in Lemma \ref{xy} then the matrix
\begin{equation*}
\begin{bmatrix} a_{11} & a_{12} & u & e^{-i{\theta}/{2}}x_1\\
a_{21} & a_{11} & v & e^{-i{\theta}/{2}}x_2\\
w^* & q^* & S & 0\\
e^{-i{\theta}/{2}}x_2 & e^{-i{\theta}/{2}}x_1 & 0 & z
\end{bmatrix}
\end{equation*}
is a minimal normal completion of $\widetilde{A}$. This
observation leads to the following statement.}
\end{rem}
\begin{lem}\label{nscond_new}
Suppose that
\begin{equation*}
\widetilde{A}=\begin{bmatrix} a_{11} & a_{12} & u\\
a_{21} & a_{11} & v\\
w^* & q^* & S
\end{bmatrix}\in\mathbb{C}^{n\times n}
\end{equation*}
 satisfies \eqref{id1'}
with $H$ given by \eqref{h}. Then $\nd(A)=1$ if and only if there
exist $x_1,x_2\in\mathbb{C}$ satisfying \eqref{cond} and such that
\begin{equation}\label{eq3'}
 u^*x_1+v^*x_2 =
w^*\overline{x}_2+q^*\overline{x}_1.
\end{equation}
In this case, the matrix
\begin{equation}\label{tilde_B_1}
\begin{bmatrix} a_{11} & a_{12} & u & x_1\\
a_{21} & a_{11} & v & x_2\\
w^* & q^* & S & 0\\
x_2 & x_1 & 0 & z
\end{bmatrix},
\end{equation}
where $z$ is given by \eqref{z_def}, is a minimal normal
completion of $\widetilde{A}$.
\end{lem}
\begin{proof}
The statement follows from Lemma \ref{xy}, Remark \ref{x_adj}, and
the observation that, if $y_1=\overline{x}_2$ and $y_2=\overline{x}_1$,
then \eqref{eq3} becomes \eqref{eq3'}. Solving \eqref{eq1} or
\eqref{eq2} (which are equivalent in this case) for $z$ gives
\eqref{z_def}.
\end{proof}
\begin{proof}[Proof of Theorem \ref{thm1}]
(i) By Corollary \ref{nec_cond} the rank condition,
$\rank(A^*A-AA^*)=2$, or equivalently, \eqref{rank_cond} with some
$d>0$, is necessary for $A$ to have normal defect one, thus we can
assume that this condition holds. Let
  $u_1$ and $u_2$ be the unit
eigenvectors of the matrix $A^*A-AA^*$ corresponding to its
eigenvalues $\lambda_1=d$ and $\lambda_2=-d$, and let $u_3$,
\ldots, $u_n$ be an arbitrary orthonormal basis of
$\Null(A^*A-AA^*)$. Define a unitary matrix $U=\begin{bmatrix} u_1
& \ldots & u_n\end{bmatrix}\in\mathbb{C}^{n\times n}$ and an
isometry $U^{\prime}=\begin{bmatrix} u_3 & \ldots & u_n
\end{bmatrix}\in\mathbb{C}^{n\times (n-2)}$. Then
\begin{equation}\label{U'}
U^{\prime}U^{\prime *}=P,
\end{equation}
and the matrix $\widetilde{A}=U^*AU$ has the form \eqref{A_form'},
where the scalars $a_{ij}$ are defined by \eqref{a's},
\begin{equation}\label{vects} u=u_1^*AU^\prime,\quad
v=u_2^*AU^\prime,\quad w^*=U^{\prime *}Au_1,\quad q^*=U^{\prime
*}Au_2.
\end{equation}
According to Lemma \ref{nscond_new}, $\nd(\widetilde{A})=1$ (and
hence $\nd(A)=1$) if and only if \eqref{eq3'} is satisfied with
some $x_1,x_2\in\mathbb{C}$ subject to \eqref{cond}. By
\eqref{vects}, equation \eqref{eq3'} can be written as
\begin{equation*}
(U^{\prime *}A^*u_1)x_1+(U^{\prime *}A^*u_2)x_2=(U^{\prime
*}Au_1)\overline{x}_2+(U^{\prime *}Au_2)\overline{x}_1.
\end{equation*}
Multiplying on the left by $U^\prime $ and taking into account
that $U^{\prime}\colon\mathbb{C}^{n-2}\to\mathbb{C}^n$ is an
isometry satisfying \eqref{U'}, we obtain an equivalent equation
\begin{equation}\label{main_id'}
\mathbf{u}^*x_1+\mathbf{v}^*x_2=\mathbf{w}^*\overline{x}_2+\mathbf{q}^*\overline{x}_1,
\end{equation}
with the vectors
\begin{equation}\label{hat_vects}
\mathbf{u}^*=PA^*u_1,\quad \mathbf{v}^*=PA^*u_2, \quad
\mathbf{w}^*=PAu_1, \quad \mathbf{q}^*=PAu_2.
\end{equation}
Note that these vectors are defined independently of the choice of
$u_3$, \ldots, $u_n$.  Since \eqref{main_id'} is equivalent to
\eqref{eq}, this proves part (i) of this theorem.

(ii) If $\nd(A)=1$ and $\widetilde{A}$ is defined as in part (i), then
$\nd(\widetilde{A})=1$. By Lemma \ref{nscond_new}, for any
$x_1,x_2\in\mathbb{C}$ satisfying \eqref{eq} (or equivalently,
\eqref{eq3'}) and \eqref{cond}, the matrix in \eqref{tilde_B_1} is a
minimal normal completion of $\widetilde{A}$. By Remark \ref{x_adj},
so is
\begin{equation}\label{tilde_B}
\widetilde{B}=\begin{bmatrix} a_{11} & a_{12} & u & \mu x_1\\
a_{21} & a_{11} & v & \mu x_2\\
w^* & q^* & S & 0\\
\overline{\mu} x_2 & \overline{\mu} x_1 & 0 & z
\end{bmatrix},
\end{equation}
with an arbitrary $\mu\in\mathbb{T}$. By Lemma \ref{xy}, all
minimal normal completions of $\widetilde{A}$ arise in this way.
Since $\widetilde{A}=U^*AU$ and $U$ is unitary, all minimal normal
completions of $A$ have the form
\begin{equation*}
B=\begin{bmatrix} U & 0\\
0 & 1
\end{bmatrix}\widetilde{B}\begin{bmatrix} U^* & 0\\
0 & 1
\end{bmatrix}=\begin{bmatrix} A & \mu(x_1u_1+x_2u_2)\\
\overline{\mu}({x}_2u_1^*+{x}_1u_2^*) & z
\end{bmatrix},
\end{equation*}
where $\widetilde{B}$ is any of the matrices defined in \eqref{tilde_B}.
This proves part (ii).
\end{proof}
\begin{proof}[Proof of Theorem \ref{normalext1}]
If $\nd(A)=1$ then there exists a normal matrix \begin{equation*}
B=\begin{bmatrix} A & x
\\ y^* & z
\end{bmatrix}\in\mathbb{C}^{(n+1)\times (n+1)}.\end{equation*}
By Lemma \ref{nec_suf_cond}, \eqref{xy_eq} holds and
$A^*x-\overline{z}x=Ay-zy$, i.e., $x,y, A^*x, Ay$ are linearly
dependent.

Conversely, suppose that $A^*A-AA^*=xx^*-yy^*$ is satisfied with
some linearly independent vectors $x,y\in\mathbb{C}^n$, and the
vectors $x,y, A^*x, Ay$ are linearly dependent. Clearly, in this case
$\rank(A^*A-AA^*)=2$. Choose an orthonormal eigenbasis
$u_1,\ldots,u_n$ of the matrix $A^*A-AA^*$ and define a unitary matrix
$U=\begin{bmatrix} u_1 & \ldots & u_n\end{bmatrix}$ as in the proof
of Theorem \ref{thm1}. By Lemma \ref{other_nec}, the matrix
$\widetilde{A}=U^*AU$ has the form \eqref{A_form'}. Define new
linearly independent vectors $\widetilde{x}=U^*x$,
$\widetilde{y}=U^*y$.  Then
$\widetilde{A}^*\widetilde{A}-\widetilde{A}\widetilde{A}^*=\widetilde{x}\widetilde{x}^*-\widetilde{y}\widetilde{y}^*$.
As in the proof of Lemma \ref{xy}, we conclude that
\eqref{xy_new}--\eqref{y_via_x} hold with $\widetilde{x}$ and
$\widetilde{y}$ in the place of $x$ and $y$. Since $x,y, A^*x, Ay$ are
linearly dependent, so are $\widetilde{x},\widetilde{y},
\widetilde{A}^*\widetilde{x}, \widetilde{A}\widetilde{y}$, i.e., the matrix
$$\begin{bmatrix}
\widetilde{x} & \widetilde{y} & \widetilde{A}^*\widetilde{x} &
\widetilde{A}\widetilde{y}
\end{bmatrix}=\begin{bmatrix}
\widetilde{x}_1 & \widetilde{y}_1 &
\overline{a}_{11}\widetilde{x}_1+\overline{a}_{21}\widetilde{x}_2
&
a_{11}\overline{y}_1+a_{12}\overline{y}_2\\
\widetilde{x}_2 & \widetilde{y}_2 &
\overline{a}_{12}\widetilde{x}_1+\overline{a}_{11}\widetilde{x}_2
&
a_{21}\widetilde{y}_1+a_{11}\widetilde{y}_2\\
0 & 0 & u^*\widetilde{x}_1+v^*\widetilde{x}_2 &
w^*\widetilde{y}_1+q^*\widetilde{y}_2
\end{bmatrix}\in\mathbb{C}^{n\times 4}$$
has rank less than $4$. Therefore
$u^*\widetilde{x}_1+v^*\widetilde{x}_2$ and
$w^*\widetilde{y}_1+q^*\widetilde{y}_2$ are linearly dependent.
The identities
$\widetilde{y_1}=e^{i\theta}\overline{\widetilde{x}}_2$ and
$\widetilde{y_2}=e^{i\theta}\overline{\widetilde{x}}_1$, together
with the first three identities in \eqref{uvqw}, imply that there
is $\phi\in\mathbb{R}$ such that
$$u^*\widetilde{x}_1+v^*\widetilde{x}_2=(w^*\widetilde{y}_1+q^*\widetilde{y}_2)e^{i\phi}=
(w^*\overline{\widetilde{x}}_2+q^*\overline{\widetilde{x}}_1)e^{i(\theta+\phi)}.$$
Putting $x_1^0=\widetilde{x}_1e^{-i\frac{\theta+\phi}{2}}$ and
$x_2^0=\widetilde{x}_2e^{-i\frac{\theta+\phi}{2}}$, we obtain
$$u^*x_1^0+v^*x_2^0=w^*\overline{x}_2^0+q^*\overline{x_1^0}.$$
By Lemma \ref{nscond_new}, $\nd(\widetilde{A})=1$, and therefore
$\nd(A)=1$.

The last statement of the theorem is obtained as follows. Let
$$x^0=\col(x_1^0,x_2^0,0,\ldots,0),\quad
y^0=\col(\overline{x}_2^0,\overline{x}_1^0,0,\ldots,0),$$ and
\begin{equation*}
z=a_{11}-\frac{1}{d}\left(x_2^0(a_{12}\overline{x}_1^0-\overline{a}_{21}x_2^0)+
x_1^0(\overline{a}_{12}x_1^0-a_{21}\overline{x}_2^0)\right)
\end{equation*}
(see \eqref{z_def}). By Lemma \ref{nscond_new}, the
matrix $\widetilde{B}=\begin{bmatrix} \widetilde{A} & x^0\\
y^{0*} & z\end{bmatrix}$ is a minimal normal completion of
$\widetilde{A}$. Then
$$B=\begin{bmatrix}
U & 0\\
0 & e^{i\frac{\phi-\theta}{2}}
\end{bmatrix}\begin{bmatrix} \widetilde{A} & x^0\\
y^{0*} & z\end{bmatrix}\begin{bmatrix}
U^* & 0\\
0 & e^{i\frac{\theta-\phi}{2}}
\end{bmatrix}=\begin{bmatrix}
A & e^{-i\phi} x\\
y^* & z\end{bmatrix}$$ is a minimal normal completion of $A$,
i.e., we obtain \eqref{nm} with $\nu=e^{-i\phi}$.
\end{proof}
Applying Theorem \ref{normalext1} to $3\times 3$ matrices, we
obtain the following.
\begin{cor}\label{n=3}
A matrix $A\in\mathbb{C}^{3\times 3}$ has normal defect one if and
only if $\rank(A^*A-AA^*)=2$. Thus, for any $3\times 3$ matrix
$A$, one has
$$\nd(A)=\max\{i_+(A^*A-AA^*),i_-(A^*A-AA^*)\},$$ i.e., \eqref{lower_bound}  in this case
turns into the equality. \end{cor}
\begin{proof}
The necessity of the rank condition has been established in Corollary
\ref{nec_cond}. The sufficiency follows from Theorem
\ref{normalext1}, since any four vectors in $\mathbb{C}^3$ are linearly
dependent. The second statement is obvious in the case where
$A^*A-AA^*=0$. It easily follows from the first statement in the case
when $\rank(A^*A-AA^*)=2$. If $\rank(A^*A-AA^*)=3$ then
$\nd(A)=\ud(A)=2$, and since $A^*A-AA^*$ has three nonzero
eigenvalues (counting multiplicities), we also have
$$\max\{i_+(A^*A-AA^*),i_-(A^*A-AA^*)\}=2.$$ This covers all
the possibilities, since $\rank(A^*A-AA^*)\neq 1$ due to the
identity $\trace(A^*A-AA^*)=0$.
\end{proof}

In the following example, we show that for $n>3$ the rank condition
(\ref{rank}) is not sufficient for $\nd(A)=1$.
\begin{ex}\label{rc_ns}
\rm{Let $$A=\begin{bmatrix} 0 & 0 & 1 & -i\\
2 & 0 & 0 & 0\\
0 & 1 & \frac{1}{\sqrt{2}} & \frac{i}{\sqrt{2}}\\
0 & -i & \frac{i}{\sqrt{2}} & -\frac{1}{\sqrt{2}}
\end{bmatrix}.$$
Note that $A\ (=\widetilde{A})$ is already of the form
\eqref{A_form'}. We have
$$A^*A-AA^*=\begin{bmatrix} 2 & 0 & 0 & 0\\
0 & -2 & 0 & 0\\
0 & 0 & 0 & 0\\
0 & 0 & 0 & 0
\end{bmatrix}.$$
Equation \eqref{eq3'} in this case takes the form
$$x_1\begin{bmatrix} 1\\
i\end{bmatrix}=\overline{x}_1\begin{bmatrix} 1\\
-i\end{bmatrix},$$ and it has no solutions with
$|x_1|^2-|x_2|^2=2>0$. Thus, by Lemma \ref{nscond_new},
$\nd(A)>1$.}
\end{ex}

\begin{rem}\label{eig}
\rm{If $\rank(A^*A-AA^*)=2$ and $u_1,u_2\in\mathbb{C}^n$ are the
unit eigenvectors of $A^*A-AA^*$ corresponding to the eigenvalues
$\lambda_1=d (>0)$ and $\lambda_2=-d$, then the vectors
$x=\sqrt{d}u_1$ and $y=\sqrt{d}u_2$ satisfy \eqref{xy_eq}. Indeed,
$u_1$ and $u_2$ are orthogonal, hence linearly independent, and
$\spn(u_1,u_2)=\range(A^*A-AA^*)$. For arbitrary $a,b\in\mathbb{C}$,
we have $$(A^*A-AA^*)(a u_1+b
u_2)=d(au_1-bu_2)=d(u_1u_1^*-u_2u_2^*)(au_1+bu_2),$$ therefore
$A^*A-AA^*=d(u_1u_1^*-u_2u_2^*)$.

However, as the following example shows, these $x$ and $y$ do not
necessarily satisfy the conditions of Theorem \ref{normalext1}. }
\end{rem}
\begin{ex}\label{ex:eig}
\rm{Let $$A=\begin{bmatrix} 0 & 0 & \frac{1}{\sqrt{2}} &
\frac{i}{\sqrt{2}}\\
0 & 0 & 1 & i\\
1 & \frac{1}{\sqrt{2}} & \frac{\sqrt{3}}{2} &
-\frac{\sqrt{3}}{2}i\\
i & \frac{i}{\sqrt{2}} & -\frac{\sqrt{3}}{2}i &
-\frac{\sqrt{3}}{2}
\end{bmatrix}.$$
Then
$$A^*A-AA^*=\begin{bmatrix} 1 & 0 & 0 & 0\\
0 & -1 & 0 & 0\\
0 & 0 & 0 & 0\\
0 & 0 & 0 & 0
\end{bmatrix}=e_1e_1^*-e_2e_2^*,$$
where $e_1$ and $e_2$ are standard basis vectors, which are the
eigenvectors of the matrix $A^*A-AA^*$ corresponding to its
eigenvalues $\lambda_1=1$ and $\lambda_2=-1$. However, the vectors
$x=e_1$, $y=e_2$,
$A^*x=\col(0,0,\frac{1}{\sqrt{2}},-\frac{i}{\sqrt{2}})$,
$Ay=\col(0,0,\frac{1}{\sqrt{2}},\frac{i}{\sqrt{2}})$,
 are linearly independent.

 On the other hand, we have
$\nd(A)=\ud(A)=1$: one of minimal normal completions of $A$ (in
fact, its minimal unitary completion) is
$$B=\begin{bmatrix} 0 & 0 & \frac{1}{\sqrt{2}} &
\frac{i}{\sqrt{2}} & \sqrt{2}\\
0 & 0 & 1 & i & -1\\
1 & \frac{1}{\sqrt{2}} & \frac{\sqrt{3}}{2} &
-\frac{\sqrt{3}}{2}i & 0\\
i & \frac{i}{\sqrt{2}} & -\frac{\sqrt{3}}{2}i &
-\frac{\sqrt{3}}{2} & 0\\
-1 & \sqrt{2} & 0 & 0 & 0
\end{bmatrix}.$$ }
\end{ex}

 We are now able
to describe a procedure to determine whether $\nd(A)=1$ for a
given matrix $A\in\mathbb{C}^{n\times n}$, i.e., whether equation
\eqref{eq} in Theorem \ref{thm1} has a solution pair
$x_1,x_2\in\mathbb{C}$ satisfying \eqref{cond}. Moreover, this
procedure allows to find all such solutions, and then, applying
part (ii) of Theorem \ref{thm1}, all minimal normal completions of
an arbitrary matrix $A$ with $\nd(A)=1$.
\medskip

\centerline{\textbf{The procedure}}

\medskip

\noindent\textbf{Begin}
\medskip

\textbf{Step 1.} Verify the condition $\rank(A^*A-AA^*)=2$, or
equivalently, \eqref{rank_cond}. If it is satisfied -- go to Step
2. Otherwise, stop: $\nd(A)>1$.

\textbf{Step 2.} Rewrite \eqref{eq} in the form \eqref{main_id'},
 where $\mathbf{u}$, $\mathbf{v}$, $\mathbf{w}$, $\mathbf{q}$ are
defined in \eqref{hat_vects} (see Theorem \ref{thm1} for the
definition of $P$, $u_1$, and $u_2$). Let
$$\mathbf{u}=\mathbf{u}_R+i\mathbf{u}_I,\
\mathbf{v}=\mathbf{v}_R+i\mathbf{v}_I,\
\mathbf{q}=\mathbf{q}_R+i\mathbf{q}_I,\
\mathbf{w}=\mathbf{w}_R+i\mathbf{w}_I,
$$
where
$\mathbf{u}^T_R,\mathbf{u}^T_I,\mathbf{v}^T_R,\mathbf{v}^T_I,\mathbf{q}^T_R,\mathbf{q}^T_I,\mathbf{w}^T_R,
\mathbf{w}^T_I\in\mathbb{R}^n$, and let
$$x_1=x_{R1}+ix_{I1},\quad x_2=x_{R2}+ix_{I2},$$
where $x_{R1}, x_{I1}, x_{R2}, x_{I2}\in\mathbb{R}$. Then
\eqref{main_id'} becomes
\begin{equation}\label{short}
\begin{bmatrix}
\mathbf{u}^T_R-\mathbf{q}^T_R & \mathbf{u}^T_I+\mathbf{q}^T_I &
\mathbf{v}^T_R-\mathbf{w}^T_R &
\mathbf{v}^T_I+\mathbf{w}^T_I\\
-\mathbf{u}^T_I+\mathbf{q}^T_I & \mathbf{u}^T_R+\mathbf{q}^T_R &
-\mathbf{v}^T_I+\mathbf{w}^T_I & \mathbf{v}^T_R+\mathbf{w}^T_R
\end{bmatrix}\begin{bmatrix}
x_{R1}\\
x_{I1}\\
x_{R2}\\
x_{I2}
\end{bmatrix}=0.
\end{equation}
Denote $$\mathbf{Q}=\begin{bmatrix} \mathbf{u}^T_R-\mathbf{q}^T_R
& \mathbf{u}^T_I+\mathbf{q}^T_I & \mathbf{v}^T_R-\mathbf{w}^T_R &
\mathbf{v}^T_I+\mathbf{w}^T_I\\
-\mathbf{u}^T_I+\mathbf{q}^T_I & \mathbf{u}^T_R+\mathbf{q}^T_R &
-\mathbf{v}^T_I+\mathbf{w}^T_I & \mathbf{v}^T_R+\mathbf{w}^T_R
\end{bmatrix}.$$
Find $m=\rank (\mathbf{Q})$.

\textbf{Step 3.} Depending on $m$, consider the following cases.
\begin{itemize}
    \item[\textbf{(1)}] $m=0$. In this case,
    $\mathbf{u}=\mathbf{v}=\mathbf{q}=\mathbf{w}=0$,
     and then \eqref{main_id'}
    holds with any $x_1,x_2\in\mathbb{C}$ such that
    $|x_1|^2-|x_2|^2=d$. Therefore,
     $\nd(A)=1$. Go to Step 4.

    \item[\textbf{(2)}] $1\le m\le 3$. In this case, \eqref{short} has
    nontrivial solutions. Let $\mathbf{F}\in\mathbb{R}^{4\times (4-m)}$ be a
   matrix whose columns are linearly independent solutions of
   \eqref{short}. Then $x=\col(x_{R1}, x_{I1},x_{R2},x_{I2}) \in\mathbb{R}^4$
   is a solution of \eqref{short} if and only if $x=\mathbf{F}h$,
   with $h\in\mathbb{R}^{4-m}$. Setting $\mathbf{F}=\begin{bmatrix}
   \mathbf{F}_1\\
   \mathbf{F}_2
   \end{bmatrix}$ where $\mathbf{F}_1,\mathbf{F}_2\in\mathbb{R}^{2\times (4-m)}$,  write the
   condition $|x_1|^2-|x_2|^2>0$ as
   $$h^T(\mathbf{F}_1^T\mathbf{F}_1-\mathbf{F}_2^T\mathbf{F}_2)h>0.$$
Therefore, $\nd(A)=1$ if and only if the matrix
$\mathbf{K}=\mathbf{F}_1^T\mathbf{F}_1-\mathbf{F}_2^T\mathbf{F}_2$
has at least one positive eigenvalue. If this is not the case,
stop. Otherwise, for any $h$ in the level hyper-surface
$h^T\mathbf{K}h=d$,  define $$\begin{bmatrix} x_{R1}\\
x_{I1}
\end{bmatrix}=\mathbf{F}_1h,\quad \begin{bmatrix} x_{R2}\\
x_{I2}
\end{bmatrix}=\mathbf{F}_2h,
$$ and thus obtain $x_1=x_{R1}+ix_{I1}$, $x_2=x_{R2}+ix_{I2}$
satisfying \eqref{main_id'} and such that $|x_1|^2-|x_2|^2=d$. Go
to Step 4.

    \item[\textbf{(3)}] $m=4$. In this case, \eqref{short}, and hence
   \eqref{main_id'}, has no nontrivial solutions, and
    $\nd(A)>1$. Stop.
\end{itemize}

\textbf{Step 4.} For each pair $x_1,x_2\in\mathbb{C}$ obtained in
Step 3, find minimal normal completions of $A$ as described by
\eqref{ncompl}--\eqref{a's}.

\medskip

\textbf{End}

\medskip

\begin{rem}\label{m=1}
\rm{If $m=1$ then $\mathbf{K}$ always has a positive eigenvalue
and $\nd(A)=1$. Indeed, in this case $\mathbf{F}$ is a full rank
matrix of size $4\times 3$. Since $\Null (\mathbf{F}_2)\neq \{
0\}$ and $\Null (\mathbf{F})=\{ 0\}$, for a nonzero vector
$h\in\Null (\mathbf{F}_2)$ we have
$$h^T\mathbf{K}h=h^T\mathbf{F}_1^T\mathbf{F}_1h>0.$$
}
\end{rem}

\begin{ex}\label{shift}
\rm{Consider the $4\times 4$ shift matrix
$$A=\begin{bmatrix}
0 & 1 & 0 & 0 \\
0 & 0 & 1 & 0\\
0 & 0 & 0 & 1\\
0 & 0 & 0 & 0
\end{bmatrix}.$$
Being a single Jordan cell, $A$ is unitarily irreducible. Since
$$A^*A-AA^*=\begin{bmatrix}
-1 & 0 & 0 & 0 \\
0 & 0 & 0 & 0\\
0 & 0 & 0 & 0\\
0 & 0 & 0 & 1
\end{bmatrix},$$
$A$ is not normal, and the rank condition is satisfied. Clearly,
$\nd(A)=\ud(A)=1$, and a minimal unitary (and thus normal)
completion of $A$ is given by
\begin{equation}\label{uc_shift}\begin{bmatrix}
0 & 1 & 0 & 0 & 0\\
0 & 0 & 1 & 0 & 0\\
0 & 0 & 0 & 1 & 0\\
0 & 0 & 0 & 0 & \zeta\\
\rho & 0 & 0 & 0 & 0
\end{bmatrix},
\end{equation}
with any $\zeta,\rho\in \mathbb{T}$. It is less obvious that any
minimal normal completion of $A$ has this form, as shown in
\cite[Proposition 2]{KW}. We will give here an alternative explanation.
Observe that $d=1$, and the nonzero eigenvalues of $A^*A-AA^*$ are
$\lambda_1=1$, $\lambda_2=-1$. Let $e_k$, $k=1,\ldots,4$, be the
standard basis vectors in $\mathbb{C}^4$. Then $u_1=e_4$ and
$u_2=e_1$ are the eigenvectors of $A^*A-AA^*$ corresponding to
$\lambda_1=1$ and $\lambda_2=-1$. We have
$$P=I_4-e_4e_4^*-e_1e_1^*=\begin{bmatrix}
0 & 0 & 0 & 0 \\
0 & 1 & 0 & 0\\
0 & 0 & 1 & 0\\
0 & 0 & 0 & 0
\end{bmatrix}.$$
Since
$$\mathbf{u}^*=PA^*u_1=0,\  \mathbf{v}^*=PA^*u_2=e_2,\ \mathbf{w}^*=PAu_1=e_3,\ \mathbf{q}^*=PAu_2=0,$$ we have
$$\mathbf{u}_R^T=\mathbf{u}_I^T=\mathbf{q}_R^T=\mathbf{q}_I^T=\mathbf{v}_I^T=\mathbf{w}_I^T=0,\
\mathbf{v}_R^T=e_2,\ \mathbf{w}_R^T=e_3,$$ and then
$$\mathbf{Q}=\begin{bmatrix}
0 & 0 & 0 & 0 \\
0 & 0 & 1 & 0\\
0 & 0 & -1 & 0\\
0 & 0 & 0 & 0\\
0 & 0 & 0 & 0 \\
0 & 0 & 0 & 1\\
0 & 0 & 0 & 1\\
0 & 0 & 0 & 0
\end{bmatrix}.$$
Since $m=\rank(\mathbf{Q})=2$, all the solutions
$x=\col(x_{R1},x_{I1},x_{R2},x_{R2})\in\mathbb{R}^4$ of the
equation $\mathbf{Q}x=0$ are given by $x=\mathbf{F}h$ where
$$\mathbf{F}=\begin{bmatrix}
\mathbf{F}_1\\
\mathbf{F}_2
\end{bmatrix}=\begin{bmatrix}
1 & 0\\
0 & 1\\
0 & 0\\
0 & 0
\end{bmatrix},$$
and $h\in\mathbb{R}^2$ is arbitrary. The matrix
$\mathbf{K}=\mathbf{F}_1^T\mathbf{F}_1-\mathbf{F}_2^T\mathbf{F}_2=I_2$
is positive definite, and therefore $\nd(A)=1$. Define, for an
arbitrary $h\in\mathbb{R}^2$ such that $h_1^2+h_2^2=1$,
$$\begin{bmatrix} x_{R1}\\
x_{I1}
\end{bmatrix}=\mathbf{F}_1h=h,\quad \begin{bmatrix} x_{R2}\\
x_{I2}
\end{bmatrix}=\mathbf{F}_2h=0,
$$ and we obtain $x_1=x_{R1}+ix_{I1}=h_1+ih_2$ and
$x_2=x_{R2}+ix_{I2}=0$ satisfying \eqref{main_id'} and such that
$|x_1|^2-|x_2|^2=1$. Then we calculate
\begin{eqnarray*}
a_{11}=u_1^*Au_1=u_2^*Au_2=0,\quad a_{12}=u_1^*Au_2=0,\quad
a_{21}=u_2^*Au_1=0,\\
z=a_{11}-\frac{1}{d}\left(x_2(a_{12}\overline{x}_1-\overline{a}_{21}x_2)+
x_1(\overline{a}_{12}x_1-a_{21}\overline{x}_2)\right)=0.
\end{eqnarray*}
Thus, all minimal normal completions of $A$ have the form
$$\begin{bmatrix}
A & \mu x_1u_1\\
\overline{\mu}x_1u_2^* & 0
\end{bmatrix}=\begin{bmatrix}
0 & 1 & 0 & 0 & 0\\
0 & 0 & 1 & 0 & 0\\
0 & 0 & 0 & 1 & 0\\
0 & 0 & 0 & 0 & \zeta\\
\rho & 0 & 0 & 0 & 0
\end{bmatrix},$$
where $\zeta=\mu x_1$ and $\rho=\overline{\mu}x_1$, with some
$\mu\in\mathbb{T}$, i.e., $\zeta$ and $\rho$ are arbitrary complex
numbers of modulus one.  We conclude that all minimal normal
completions of $A$ have the form \eqref{uc_shift}, i.e., are
minimal unitary completions of $A$.

We note that a similar argument can be applied to weighted shift
matrices of any size $n\ge 4$, with all the weights of modulus one,
thus recovering the result of Proposition 2 in \cite{KW}. We also note
that, as was mentioned in \cite{KW}, for $n=2$ or 3 there exist
non-unitary minimal normal completions of such weighted shift
matrices. Our procedure gives the full description of these
completions $B$. Namely, for $n=2$
$$B=\begin{bmatrix}
0 & 1 & \mu x_2\\
0 & 0 & \mu x_1\\
\overline{\mu} x_1 & \overline{\mu} x_2 & x_2^2+x_1\overline{x}_2
\end{bmatrix},$$
where $\mu\in\mathbb{T}$ and $x_1,x_2\in\mathbb{C}\colon
|x_1|^2-|x_2|^2=1$ are arbitrary, and for $n=3$
$$B=\begin{bmatrix}
0 & 1 & 0 & \mu x_2\\
0 & 0 & 1 & 0\\
0 & 0 & 0 & \mu x_1\\
\overline{\mu} x_1 & 0 & \overline{\mu} x_2 & 0
\end{bmatrix},$$
where $\mu\in\mathbb{T}$ and $x_1\in\mathbb{C},
x_2\in\mathbb{R}\colon |x_1|^2-x_2^2=1$ are arbitrary.}
\end{ex}

Later on, we will present more examples of application of this
method, see Examples \eqref{ex:a} and \eqref{ex:b}.

\subsection{The generic case.}\label{finer}
The procedure described in Section \ref{shorter}, which is based on
part (i) of Theorem \ref{thm1}, allows to check whether a given matrix
$A\in\mathbb{C}^{n\times n}$ has normal defect one, and if this is the
case --- to solve the system of equations \eqref{eq}--\eqref{cond}. Part
(ii) of Theorem \ref{thm1} describes all the minimal normal
completions of $A$.
 That procedure verifies first the necessary rank condition,
and then uses only the two nonzero eigenvalues, $\lambda_1=d$ and
$\lambda_2=-d$, and the two corresponding unit eigenvectors, $u_1$
and $u_2$, of $A^*A-AA^*$. The vector equation \eqref{short}
appearing in that procedure is equivalent to a system of $2n$ real
scalar linear equations with $4$ unknowns.

In this section, we show how the procedure in Section \ref{shorter}
can be refined by using a special choice of the eigenbasis for the
matrix $A^*A-AA^*$, i.e., a special construction of orthonormal
eigenvectors $u_3$, \ldots, $u_n$ corresponding to the zero
eigenvalue. This additional analysis is rewarded by obtaining a system
of $n-2$ (as opposed to $2n$) real linear equations with $4$
unknowns. Moreover, it allows us to describe the generic situation
under the assumption that the rank condition is satisfied. The refined
procedure is based on the following theorem (the proof of which is
given later in this section).
\begin{thm}\label{uv}
Let $A\in\mathbb{C}^{n\times n}$ satisfy the rank condition
\eqref{rank} (or equivalently \eqref{rank_cond}), and let $u_1$ and
$u_2$ be the unit eigenvectors of the matrix $A^*A-AA^*$
corresponding to its eigenvalues $\lambda_1=d (>0)$ and
$\lambda_2=-d$. Then
\begin{itemize}
    \item[(i)] There exist orthonormal vectors
$u_3,\ldots,u_n\in\Null(A^*A-AA^*)$ (and thus the matrix
$W=\begin{bmatrix} u_1 & \ldots &
u_n\end{bmatrix}\in\mathbb{C}^{n\times n}$ is unitary) such that
the matrix $\widetilde{A}=W^*AW$ has the form
\begin{equation}\label{A_newest}
\widetilde{A}=\begin{bmatrix}
a_{11} & a_{12} & {u}\\
a_{21} & a_{11} & {v}\\
{v}^T & {u}^T & {S}
\end{bmatrix},
\end{equation}
with $a_{ij}$'s defined in \eqref{a's}.
    \item[(ii)] $\nd(A)=1$ if and only if
the equation
\begin{equation}\label{id_new}
\Im({u}^*x_1+{v}^*x_2)=0
\end{equation}
has a solution pair $x_1,x_2\in\mathbb{C}$ satisfying
\begin{equation}\label{cond'}
|x_1|^2-|x_2|^2=d.
\end{equation}
\item[(iii)] If $\nd(A)=1$, $x_1$ and $x_2$ satisfy \eqref{id_new} and
    \eqref{cond'}, and $\mu\in\mathbb{T}$ is arbitrary then the matrix
    $B$ defined in \eqref{ncompl} is a minimal normal completion of
    $A$.\end{itemize} All minimal normal completions of $A$ arise in
this way. \end{thm}
\begin{rem}\label{expl_W}
\rm{The matrix $W$ in Theorem \ref{uv} can be constructed explicitly
as will become clear from the proof of the theorem.}
\end{rem}

Let $A\in\mathbb{C}^{n\times n}$ satisfy the rank condition.  We
define the vectors
$\mathbf{u}^*,\mathbf{v}^*,\mathbf{w}^*,\mathbf{q}^*\in\range(P)\subset\mathbb{C}^{n}$
by \eqref{hat_vects} (see also Theorem \ref{thm1} for the definition of
$P$, $u_1$, and $u_2$). Since these vectors can be viewed as the
images of vectors $u^*,v^*,w^*,q^*\in\mathbb{C}^{n-2}$ under an
isometry so that \eqref{uvqw} holds (see Lemma \ref{other_nec} and
the proof of Theorem \ref{thm1}), we have
\begin{equation}\label{uvqw_hat}
\mathbf{u}\mathbf{u}^*=\mathbf{q}\mathbf{q}^*,\quad
\mathbf{v}\mathbf{v}^*=\mathbf{w}\mathbf{w}^*,\quad
\mathbf{u}\mathbf{v}^*=\mathbf{w}\mathbf{q}^*,\quad
\mathbf{u}\mathbf{w}^*=\mathbf{v}\mathbf{q}^*.
\end{equation}
The first three equalities mean that the linear operator
\begin{equation}\label{sym_u}
X \colon \spn(\mathbf{u}^T,\mathbf{v}^T)\longrightarrow
\spn(\mathbf{q}^*,\mathbf{w}^*)
\end{equation}
defined via
\begin{equation}\label{def_X}
 X \colon\mathbf{u}^T\longmapsto\mathbf{q}^*,\quad
\mathbf{v}^T\longmapsto\mathbf{w}^*
\end{equation}
 is a well defined unitary operator. In order to
interpret the last equality in \eqref{uvqw_hat} we need an
intermission for some definitions and results on (complex)
symmetric operators and matrices (see, e.g., \cite[Section
4.4]{HJ}).

For a subspace $\mathcal{H}$ in $\mathbb{C}^k$, denote its complex
dual  by $$\overline{\mathcal{H}}:=\{ h\in\mathbb{C}^k\colon
\overline{h}\in\mathcal{H}\}.$$  We will say that a
$\mathbb{C}$-linear operator
$L\colon\mathcal{H}\to\overline{\mathcal{H}}$ is \emph{symmetric}
if $h^TLg=g^TLh$ (or, equivalently, $\langle
Lg,\overline{h}\rangle=\langle Lh,\overline{g}\rangle$ in the
standard inner product in $\mathbb{C}^k$) for all
$g,h\in\mathcal{H}$. It is clear that a $\mathbb{C}$-linear
operator $L\colon\mathbb{C}^k\to\mathbb{C}^k$ is symmetric if and
only if its matrix in a standard basis of $\mathbb{C}^k$ is
complex symmetric, i.e, $L=L^T$. In general, a $\mathbb{C}$-linear
operator $L\colon\mathcal{H}\to\overline{\mathcal{H}}$ is
symmetric if and only if its matrix in any pair of orthonormal
bases $\mathcal{B}=\{h_j\}_{j=1}^k\subset\mathcal{H}$ and
$\overline{\mathcal{B}}=\{\overline{h}_j\}_{j=1}^k\subset\overline{\mathcal{H}}$
is complex symmetric, i.e.,  $\langle
Lh_j,\overline{h}_i\rangle=\langle Lh_i,\overline{h}_j\rangle,\
i,j=1,\ldots,k.$

We can restate this also in a coordinate-free form. Let $\mathcal{G}$
and $\mathcal{H}$ be two subspaces in $\mathbb{C}^k$. For a
$\mathbb{C}$-linear operator
$L\colon\mathcal{G}\to\overline{\mathcal{H}}$ we define its
transpose as the unique $\mathbb{C}$-linear operator
$L^T\colon\mathcal{H}\to\overline{\mathcal{G}}$ which satisfies
$h^TLg=g^TL^Th$ (or, equivalently, $\langle
Lg,\overline{h}\rangle=\langle Lh,\overline{g}\rangle$) for all
$g\in\mathcal{G}$, $h\in\mathcal{H}$). Explicitly,
$L^Th=\overline{L^*\overline{h}}$ for every $h\in\mathcal{H}$. This, in
particular, implies that $(LM)^T=M^TL^T$ for two $\mathbb{C}$-linear
operators $L$, $M$. If one interprets a vector $h\in\mathcal{H}$ as an
operator $h\colon \mathbb{C}\to\mathcal{H}$ then its transpose
$h^T$ can be interpreted as the operator
$h^T\colon\overline{\mathcal{H}}\to\mathbb{C}$, and then the
identity $h^TLg=g^TL^Th$ can be viewed also as
$g^TL^Th=(h^TLg)^T$. Finally, a $\mathbb{C}$-linear operator
$L\colon\mathcal{H}\to\overline{\mathcal{H}}$ is symmetric if and
only if $L=L^T$.

We also observe that a matrix of a $\mathbb{C}$-linear operator
$L\colon\mathcal{G}\to\overline{\mathcal{H}}$ in a pair of
orthonormal bases $\mathcal{B}_1$ and $\overline{\mathcal{B}}_2$
and a matrix of $L^T\colon\mathcal{H}\to\overline{\mathcal{G}}$ in
the pair of orthonormal bases $\mathcal{B}_2$ and
$\overline{\mathcal{B}}_1$ are transposes of each other.
\begin{lem}\label{symext}
Let $\mathcal{H}$ and $\mathcal{L}$ be subspaces in $\mathbb{C}^k$
such that $\mathcal{H}\subset\mathcal{L}$, and let $Y\colon
\mathcal{H}\to\overline{\mathcal{L}}$ be an isometry such that
$P_{\overline{\mathcal{H}}}Y\colon\mathcal{H}\to\overline{\mathcal{H}}$
is a symmetric operator. Then there exists a unitary and symmetric
operator $\widetilde{Y}\colon\mathcal{L}\to\overline{\mathcal{L}}$
such that $\widetilde{Y}|_{\mathcal{H}}=Y$.
\end{lem}
\begin{proof}
We have $$\mathcal{L}=\mathcal{H}\oplus \mathcal{G}\oplus
\mathcal{K},$$ where $\mathcal{G}=\overline{\range
(P_{\overline{L}\ominus\overline{H}}Y)}$ and
$\mathcal{K}=\mathcal{L}\ominus(\mathcal{H}\oplus\mathcal{G})$.
The Takagi decomposition (see, e.g., \cite{HJ}) of the symmetric
operator $P_{\overline{\mathcal{H}}}Y$, in a coordinate-free form,
is
$$P_{\overline{\mathcal{H}}}Y=U\Sigma U^T,$$ where
$U\colon \mathbb{C}^{\dim(\mathcal{H})}\to\overline{\mathcal{H}}$
is a unitary operator such that $Ue_1$, \ldots,
$Ue_{\dim(\mathcal{H})}$ are the eigenvectors of
$P_{\overline{\mathcal{H}}}Y\overline{P_{\overline{\mathcal{H}}}Y}$,
and $\Sigma\colon\mathbb{C}^{\dim
(\mathcal{H})}\to\mathbb{C}^{\dim (\mathcal{H})}$ is an operator
whose matrix in the standard basis of $\mathbb{C}^{\dim
(\mathcal{H})}$ is diagonal, with the singular values of
$P_{\overline{\mathcal{H}}}Y$  on the diagonal. The operator
$P_{\overline{\mathcal{G}}}Y$ can be represented as
$$P_{\overline{\mathcal{G}}}Y=V(I_{\mathbb{C}^{\dim
(\mathcal{H})}}-\Sigma^2)^{1/2}U^T,$$ where
$V\colon\mathbb{C}^{\dim (\mathcal{H})}\to\overline{\mathcal{G}}$
is a coisometry, with
$$V|_{\range(I_{\mathbb{C}^{\dim
(\mathcal{H})}}-\Sigma^2)}\colon\range(I_{\mathbb{C}^{\dim
(\mathcal{H})}}-\Sigma^2) \to\overline{\mathcal{G}}$$ unitary.
Define the operator $J\colon \mathcal{K}\to\overline{\mathcal{K}}$
for some pair of orthonormal bases
$\mathcal{B}=\{\kappa_j\}_{j=1}^{\dim(\mathcal{K})}\subset\mathcal{K}$,
$\overline{\mathcal{B}}=\{\overline{\kappa}_j\}_{j=1}^{\dim(\mathcal{K})}\subset\overline{\mathcal{K}}$
as
$$J\left(\sum_{j=1}^{\dim(\mathcal{K})}\alpha_j\kappa_j\right)=
\sum_{j=1}^{\dim(\mathcal{K})}\alpha_j\overline{\kappa}_j,\quad
\alpha_1,\ldots,\alpha_{\dim(\mathcal{K})}\in\mathbb{C}.$$ Clearly,
$J$ is symmetric, and the matrix of $J$ in the pair of bases
$\mathcal{B}$ and $\overline{\mathcal{B}}$ is
$I_{\dim(\mathcal{K})}$. It is straightforward to verify that the operator
$$\widetilde{Y}:=\begin{bmatrix}
U\Sigma U^T & U(I_{\mathbb{C}^{\dim
(\mathcal{H})}}-\Sigma^2)^{1/2}V^T & 0\\
V(I_{\mathbb{C}^{\dim (\mathcal{H})}}-\Sigma^2)^{1/2}U^T &
-V\Sigma
V^T & 0\\
0 & 0 & J
\end{bmatrix}\colon \begin{array}{c}
  \mathcal{H} \\
  \oplus \\
  \mathcal{G} \\
  \oplus \\
  \mathcal{K} \\
\end{array}\to  \begin{array}{c}
  \overline{\mathcal{H}} \\
  \oplus \\
  \overline{\mathcal{G}} \\
  \oplus \\
  \overline{\mathcal{K}} \\
\end{array}$$
has the desired properties.
\end{proof}
\begin{cor}\label{Xext}
The unitary operator $X$ defined by \eqref{sym_u}--\eqref{def_X}
can be extended to a unitary and symmetric operator
$\widetilde{X}\colon \range(\overline{P})\to\range(P)$.
\end{cor}
\begin{proof}
Since $\mathbf{u}^T,\mathbf{v}^T\in\range(\overline{P})$ and
$\mathbf{q}^*,\mathbf{w}^*\in\range(P)$ (see \eqref{hat_vects}),
the operator $X$ can be viewed as an isometry $X\colon\spn
(u^T,v^T)\to\range(P)$. The last identity in \eqref{uvqw_hat}
means that the operator $$P_{\spn
(\mathbf{u}^*,\mathbf{v}^*)}X\colon\spn
(\mathbf{u}^T,\mathbf{v}^T)\to\spn (\mathbf{u}^*,\mathbf{v}^*)$$
is symmetric. Then the statement of this corollary follows from
Lemma \ref{symext}, where we set $k=n-2$,
$$\mathcal{H}=\spn
(\mathbf{u}^T,\mathbf{v}^T)=\spn(\overline{P}A^T\overline{u}_1,\overline{P}A^T\overline{u}_2),$$
$Y=X$ and $\mathcal{L}=\range(\overline{P})$.
\end{proof}
\begin{rem}\label{aui}
\rm{It can be shown that the unitary and symmetric operator
$\widetilde{X}$ in Corollary \ref{Xext} can be constructed
bypassing Lemma \ref{symext} and using instead the following
remarkable theorem from \cite{Woron}: If $A\in\mathbb{C}^{n\times
n}$ and $x,y\in\mathbb{C}^n$ are such that $A^*A-AA^*=xx^*-yy^*$
then there exists an antiunitary involution $\iota $ on
$\mathbb{C}^n$ such that $\iota x=y$ and $\iota A\iota=A^*$. (A
mapping $\iota\colon \mathbb{C}^n\to\mathbb{C}^n$ is called an
antiunitary involution if $\iota^2=I_n$ and $\langle \iota h,\iota
g\rangle =\langle g,h\rangle$ for every $h,g\in\mathbb{C}^n$, in
the standard inner product in $\mathbb{C}^n$.) Our Lemma
\ref{symext} seems to be of independent interest, and can be
applied to other problems as well. }
\end{rem}

\begin{proof}[Proof of Theorem \ref{uv}]
(i) The operator $\widetilde{X}$ in Corollary \ref{Xext}, which is
constructed as in Lemma \ref{symext}
 for the given matrix
$A$, is symmetric and unitary, and thus has a Takagi factorization
(see \cite{HJ}) $\widetilde{X}=GG^T$ where
$G\colon\mathbb{C}^{n-2}\to\range(P)$ is unitary. One can view $G$
as an isometry $G^\prime\colon\mathbb{C}^{n-2}\to\mathbb{C}^n$
with the same range as $G$: $$\range(G^\prime
)=\range(G)=\range(P).$$ Clearly, the columns $u_3$, \ldots, $u_n$
of the standard matrix of $G^\prime$ are orthonormal, and hence,
together with $u_1$ and $u_2$, form an orthonormal eigenbasis of
$A^*A-AA^*$. We also have $G^\prime G^{\prime *}=P$. We then
extend $\widetilde{X}$ to the operator $X^\prime=G^\prime
G^{\prime T}\colon\mathbb{C}^n\to\mathbb{C}^n$.
 The operator represented
by the matrix $A$ in the standard basis of $\mathbb{C}^n$, in the
basis $u_1$, \ldots, $u_n$ has the matrix
$$\widetilde{A}=\begin{bmatrix}
u_1^*\\
u_2^*\\
G^{\prime *}
\end{bmatrix}A\begin{bmatrix} u_1 & u_2 & G^\prime \end{bmatrix}=
\begin{bmatrix}
a_{11} & a_{12} & u_1^*AG^\prime\\
a_{21} & a_{11} & u_2^*AG^\prime\\
G^{\prime *}Au_1 & G^{\prime *}Au_2 & G^{\prime *}AG^\prime
\end{bmatrix}.$$
We have
\begin{eqnarray*}
(u_1^*AG^\prime)^T &=& G^{\prime T}A^T\overline{u}_1=G^{\prime
*}G^\prime G^{\prime T}A^T\overline{u}_1=G^{\prime
*}X^{\prime}A^T\overline{u}_1=G^{\prime
*}X^{\prime}\overline{P}A^T\overline{u}_1\\
&=& G^{\prime *}PAu_2=G^{\prime *}Au_2,
\end{eqnarray*}
and  similarly,
$$(u_2^*AG^\prime)^T=G^{\prime *}Au_1.$$
Setting $u=u_1^*AG^\prime$, $v=u_2^*AG^\prime$, $S=G^{\prime
*}AG^\prime$, and $W=\begin{bmatrix} u_1 & u_2 & u_3 & \ldots &
u_n \end{bmatrix}$, we see that
 $\widetilde{A}=W^*AW$ has the form \eqref{A_newest}, which proves
part (i).

(ii) It follows from Lemma \ref{nscond_new} that
$\nd(A)=\nd(\widetilde{A})=1$ if and only if there exist
$x_1,x_2\in\mathbb{C}$ satisfying \eqref{cond'} such that
$$
{u}^*x_1+{v}^*x_2={v}^T\overline{x}_2+{u}^T\overline{x}_1.
$$
Since the last equation is equivalent to \eqref{id_new}, this
proves part (ii).

(iii) This part is proved in the same way as part (ii) of Theorem
\ref{thm1}.
\end{proof}
Let $u,v\in\mathbb{C}^{n-2}$ be as in Theorem \ref{uv},
${u}={u}_R+i{u}_I$, ${v}={v}_R+i{v}_I$, where
${u}_R^T,{u}_I^T,{v}_R^T,{v}_I^T\in\mathbb{R}^{n-2}$. Let
$x_1=x_{R1}+ix_{I1}$, $x_2=x_{R2}+ix_{I2}$, where $x_{R1}$,
$x_{I1}$, $x_{R2}$, $x_{I2}\in\mathbb{R}$. Then \eqref{id_new} can
be written as \begin{equation}\label{re_form} Qx=0,
\end{equation}
where
\begin{equation}\label{Q}
Q=\begin{bmatrix} -{u}_I^T & {u}_R^T & -{v}_I^T & {v}_R^T
\end{bmatrix}\in\mathbb{R}^{(n-2)\times 4},\quad x=\begin{bmatrix}
x_{R1}\\
x_{I1}\\
x_{R2}\\
x_{I2}
\end{bmatrix}\in\mathbb{R}^4.
\end{equation}
\begin{rem}\label{finer_proc}
\rm{Observe that replacing $\mathbf{u}$, $\mathbf{v}$,
$\mathbf{w}$, and $\mathbf{q}$  in \eqref{short} by ${u}$, ${v}$,
$\overline{v}$, and $\overline{u}$ as in Theorem \ref{uv}
 we obtain an equivalent
condition, i.e.,  equation \eqref{re_form} replaces \eqref{short}
with the matrix $Q$ replacing $\mathbf{Q}$. Thus instead of $2n$
real linear equations with $4$ unknowns we obtain $n-2$ real
linear equations with $4$ unknowns.
 Let $m=\rank (Q)$. Then, for all possible cases of $m$, the procedure for checking
 whether $\nd(A)=1$, and if this is the case -- for constructing all
 minimal normal completions of  $A$, is exactly
 the same as described is Section \ref{shorter}, with $Q$ in the place of
 $\mathbf{Q}$.}
\end{rem}

 We describe now the generic situation for each $n$,
 under the assumption that $\rank(A^*A-AA^*)=2$. In other words,
 we obtain certain topological characterization of the set of
 matrices with normal defect one in each matrix dimension.

\medskip

\centerline{\textbf{The generic case}}

\medskip

Let $A\in\mathbb{C}^{n\times n}$ satisfy the rank condition.
Consider the following possibilities for the value of $n$, and
describe the situation for each case separately.

 \begin{description}
    \item[$\mathbf{n=2}$ or $\mathbf{n=3}$] Vectors ${u}$, ${v}$ as
        in Theorem \ref{uv} do not arise (when $n=2$) or are scalars
        (when $n=3$). Then $m=\rank(Q)\le 1$, and $\nd(A)=1$
(for the case where $m=1$ it follows from Remark \ref{m=1}). This
gives a new proof of the statement on $2\times 2$ matrices in
Section \ref{sec:intro} and of
    Corollary \ref{n=3}.
    \item[$\mathbf{n=4}$ or $\mathbf{n=5}$] In these cases,  $m\le
        2$ (resp., $m\le 3$). Thus, equation \eqref{re_form} (see also
        \eqref{Q}) has nontrivial solutions. Both the situation where the
        matrix $\mathbf{K}$, constructed from $Q$ instead of
        $\mathbf{Q}$, has at least one positive eigenvalue (in which
        case $\nd(A)=1$) and where $\mathbf{K}$ has no positive
        eigenvalues (in which case $\nd(A)>1$) occur on sets with
        nonempty interior in the relative topology of the manifolds
        ${\frak M}_4$ and ${\frak M}_5$ (see page \pageref{manif} for
        the definition of ${\frak M}_n$).
    \item[$\mathbf{n\ge 6}$] In this case, generically $m=4$, thus
        \eqref{re_form} has no nontrivial solutions. Therefore,
        generically $\nd(A)>1$. Still, there are matrices $A$ with
        $\nd(A)=1$, which can be constructed, e.g., using Theorem
        \ref{nd1} and Remark \ref{not_normal}.
 \end{description}

\subsection{Normal defect and unitary defect}\label{exs} In this section, we give
two examples which show that the question in \cite{W} (see also
\cite{KW}) asking whether $\nd(A)=\ud(A)$ for any unitarily
irreducible matrix $A$ has a negative answer. In the first
example, $A$ has a single cell in its Jordan form, and in the
second example, $A$ has three distinct eigenvalues. We also
present all minimal normal completions of $A$ in both examples.
\begin{ex}\label{ex:a}\rm{
Let $$A=\begin{bmatrix} 1 & 0 & 0\\
0 & 1 & 1\\
1 & 0 & 1
\end{bmatrix}.$$
Then $$A^*A=\begin{bmatrix} 2 & 0 & 1\\
0 & 1 & 1\\
1 & 1 & 2
\end{bmatrix},\quad AA^*=\begin{bmatrix} 1 & 0 & 1\\
0 & 2 & 1\\
1 & 1 & 2
\end{bmatrix},$$
and $$ A^*A-AA^*=\begin{bmatrix} 1 & 0 & 0\\
0 & -1 & 0\\
0 & 0 & 0
\end{bmatrix}.$$
Clearly, the rank condition holds. It follows from Corollary \ref{n=3}
that $\nd(A)=1$. Since the only eigenvalue of $A$ is $1$, and $A-I$ is
nilpotent of order $3$, $A$ has a single cell in its Jordan form, and
hence $A$ is unitarily irreducible. The characteristic polynomial of
$A^*A$ is
$$p(\lambda)=(2-\lambda)^2(1-\lambda)+2\lambda-3.$$
We have $p(0)=1>0$, $p(1)=-1<0$, $p(2)=1>0$, and $p(4)=-7<0$.
Therefore, $p(\lambda)$ has three distinct roots, in intervals
$(0,1)$, $(1,2)$, and $(2,4)$, i.e., $A$ has three distinct
singular values. Therefore, $\ud(A)=2$.

We also observe that $A$ has the form \eqref{A_newest}. The
procedure described in Section \ref{shorter} together with Theorem
\ref{thm1} (or its refined version described in Remark
\ref{finer_proc} together with Theorem \ref{uv}) gives that all
minimal normal completions of $A$ have the form
$$B=\begin{bmatrix}
1 & 0 & 0 & \mu x_1\\
0 & 1 & 1 & \mu x_2\\
1 & 0 & 1 & 0\\
\overline{\mu} x_2 & \overline{\mu} x_1 & 0 & 1
\end{bmatrix},$$
with arbitrary $\mu\in\mathbb{T}$, and $x_1\in\mathbb{C}$,
$x_2\in\mathbb{R}$ satisfying $|x_1|^2-x_2^2=1$.
 }
\end{ex}
\begin{ex}\label{ex:b} \rm{
Let $$A=\begin{bmatrix} 0 & 1 & 0\\
1 & 0 & 1\\
0 & 1 & \frac{3}{2}i
\end{bmatrix}.$$
Changing the basis, we obtain $\widetilde{A}=U^*AU$, where
$$U=\begin{bmatrix} 0 & 0 & 1\\
\frac{1}{\sqrt{2}} & \frac{1}{\sqrt{2}} & 0\\
-\frac{i}{\sqrt{2}} & \frac{i}{\sqrt{2}} & 0
\end{bmatrix}$$ is unitary
and $$\widetilde{A}=\begin{bmatrix} \frac{3i}{4} & \frac{i}{4} & \frac{1}{\sqrt{2}}\\
-\frac{7i}{4} & \frac{3i}{4} & \frac{1}{\sqrt{2}}\\
\frac{1}{\sqrt{2}} & \frac{1}{\sqrt{2}} & 0
\end{bmatrix}$$
satisfies
$$\widetilde{A}^*\widetilde{A}-\widetilde{A}\widetilde{A}^*=\begin{bmatrix} 3 & 0 & 0\\
0 & -3 & 0\\
0 & 0 & 0
\end{bmatrix}.$$
Clearly,
$\rank(A^*A-AA^*)=\rank(\widetilde{A}^*\widetilde{A}-\widetilde{A}\widetilde{A}^*)=2$.
By Corollary \ref{n=3}, $\nd(A)=1$. We also observe that
$\widetilde{A}$ has the form \eqref{A_newest}. The matrix $A$ is
unitarily irreducible. Indeed, matrices
$$ \Re (A)=\begin{bmatrix} 0 & 1 & 0\\
1 & 0 & 1\\
0 & 1 & 0
\end{bmatrix}\quad \rm{and}\quad \Im (A)=\begin{bmatrix} 0 & 0 & 0\\
0 & 0 & 0\\
0 & 0 & \frac{3}{2}
\end{bmatrix}$$ do not have common eigenvectors.

Next we show that $\ud(A)=2$. The characteristic polynomial of
 $A^*A$,
$$p(\lambda)=(1-\lambda)(2-\lambda)\left(\frac{13}{4}-\lambda\right)+\frac{13}{4}\lambda-\frac{17}{4},$$
has values $p(0)=\frac{9}{4}>0$, $p(1)=-1<0$,
$p(2)=\frac{9}{4}>0$, $p(5)=-9<0$. Therefore, $p(\lambda)$ has
three distinct roots, in intervals $(0,1)$, $(1,2)$, and $(2,5)$,
i.e., $A$  has three distinct singular values. Thus, $\ud(A)=2$.
Note  that in this example $A$ has three distinct eigenvalues,
$\lambda_1=i$, $\lambda_2=\frac{\sqrt{23}+i}{4}$,
$\lambda_3=\frac{-\sqrt{23}+i}{4}$.

The procedure described in Section \ref{shorter} together with
Theorem \ref{thm1} (or its refined version described in Remark
\ref{finer_proc} together with Theorem \ref{uv}) gives that all
 minimal normal completions of $\widetilde{A}$ have the form
$$\widetilde{B}=\begin{bmatrix} \frac{3i}{4} & \frac{i}{4} & \frac{1}{\sqrt{2}} & \mu (h_1+ih_3)\\
-\frac{7i}{4} & \frac{3i}{4} & \frac{1}{\sqrt{2}} & \mu (h_2-ih_3)\\
\frac{1}{\sqrt{2}} & \frac{1}{\sqrt{2}} & 0 & 0\\
\overline{\mu}(h_2-ih_3) & \overline{\mu}(h_1+ih_3) & 0 &
\frac{h_1h_3+5h_2h_3+i(3-2h_1h_2+2h_2^2)}{3}
\end{bmatrix},$$
with arbitrary $\mu\in\mathbb{T}$, and
$h_1,h_2,h_3\in\mathbb{R}\colon h_1^2-h_2^2=3$. Correspondingly,
all minimal normal completions of $A$ have the form
$$B=\begin{bmatrix} 0 & 1 & 0 & 0\\
1 & 0 & 1 & \mu\frac{h_1+h_2}{\sqrt{2}}\\
0 & 1 & \frac{3}{2}i & \mu\frac{2h_3+i(h_2-h_1)}{\sqrt{2}}\\
0 & \overline{\mu}\frac{h_1+h_2}{\sqrt{2}} &
\overline{\mu}\frac{2h_3+i(h_2-h_1)}{\sqrt{2}} &
\frac{h_1h_3+5h_2h_3+i(3-2h_1h_2+2h_2^2)}{3}
\end{bmatrix}.$$
}
\end{ex}

\section{The real case}\label{sec:re}
Let $A\in\mathbb{R}^{n\times n}$. We define the \emph{real normal
defect of $A$}, $\rnd(A)$, as the minimal possible nonnegative
integer $p$
such that a matrix $\begin{bmatrix} A & *\\
* & *
\end{bmatrix}\in\mathbb{R}^{(n+p)\times (n+p)}$ is normal (such a matrix with the minimal possible $p$ is
a \emph{minimal real normal completion of $A$}). It is clear that
$\rnd(A)\ge\nd(A)$. We will show later (in Corollary \ref{nd_rnd})
that $\nd(A)=1$ if and only if $\rnd(A)=1$, while we leave open
the question whether $\rnd(A)=\nd(A)$ is always the case.

We also define the \emph{orthogonal defect of $A$} as the minimal
nonnegative integer $s$
such that a matrix $\begin{bmatrix} A & *\\
* & *
\end{bmatrix}\in\mathbb{R}^{(n+s)\times (n+s)}$ is a multiple of
an orthogonal matrix (such a matrix with the minimal possible $s$ is a
\emph{minimal orthogonal completion of $A$}). In fact, the
orthogonal defect of $A$ coincides with $\ud(A)$, so that it does not
require a separate notation. Indeed, since a minimal orthogonal completion of a real matrix is
obtained using the same construction as for a minimal unitary
completion (see \cite{W}), the only difference being that  the  real SVD
is involved, the size of this minimal orthogonal completion is the same
as for a minimal unitary completion.

It is clear that $\rnd(A)\le\ud(A)$, and we will show later, in Example
\ref{real_ex:a}, that there are orthogonally irreducible matrices $A$
such that the strict inequality takes place.

\subsection{Construction of real matrices of even size with real normal defect
one}\label{re_constr} The following theorem is an analogue of
Theorem \ref{nd1} for the case of real $n\times n $ matrices with
$n$ even.
\begin{thm}\label{rnd1}
Let $A\in\mathbb{R}^{n\times n}$, where $n=2k$, be not normal. The
following statements are equivalent:
\begin{itemize}
    \item[(i)] $\rnd(A)= 1$.
    \item[(ii)] There exist a contraction matrix $C\in\mathbb{R}^{n\times n}$ with $\ud(C)=1$, a block diagonal
    matrix $D\in\mathbb{R}^{n\times n}$ of the form
\begin{equation}\label{bd}
D=\diag\left(\begin{bmatrix} \alpha_1 & \beta_1\\
-\beta_1 & \alpha_1
\end{bmatrix},\ldots,\begin{bmatrix} \alpha_\ell & \beta_\ell\\
-\beta_\ell & \alpha_\ell
\end{bmatrix},\alpha_{2\ell+1},\ldots,\alpha_{2k}
 \right),
\end{equation}
     and a scalar $\mu\in\mathbb{R}$
    such that
    \begin{equation}\label{repr1_real}
    A=CDC^T+\mu I_n.
    \end{equation}
    \item[(iii)] There exist an orthogonal matrix $V\in\mathbb{R}^{n\times n}$,
    a normal matrix $N\in\mathbb{R}^{n\times n}$, and scalars
 $t,\mu\in\mathbb{R}$, with
    $0\le t<1$, such that
    \begin{equation}\label{repr2_real}
    V^TAV=MNM+\mu I_n,
    \end{equation}
where $M=\diag(1,\ldots,1,t)$.
\end{itemize}
\end{thm}
\begin{proof}
(i)$\Longleftrightarrow$(ii) Let $\rnd(A)= 1$, and let
$\begin{bmatrix} A & x\\
y^T & z
\end{bmatrix}\in\mathbb{R}^{(n+1)\times(n+1)}$ be a minimal real normal
completion of $A$. Then (see, e.g., \cite[Section IX.13]{G}) there
exist a block diagonal matrix $\Lambda\in\mathbb{R}^{n\times n}$
of the form
\begin{equation*}
\Lambda=\diag\left(\begin{bmatrix} \mu_1 & \beta_1\\
-\beta_1 & \mu_1
\end{bmatrix},\ldots,\begin{bmatrix} \mu_\ell & \beta_\ell\\
-\beta_\ell & \mu_\ell
\end{bmatrix},\mu_{2\ell+1},\ldots,\mu_{2k}
 \right),
\end{equation*}
 a scalar $\mu\in\mathbb{R}$,
and an orthogonal matrix $O=\begin{bmatrix} O_{11} & O_{12}\\
O_{21} & O_{22}
\end{bmatrix}\in\mathbb{R}^{(n+1)\times(n+1)}$ such that
\begin{equation}\label{diag_real}
\begin{bmatrix} A & x\\
y^T & z
\end{bmatrix}=\begin{bmatrix} O_{11} & O_{12}\\
O_{21} & O_{22}
\end{bmatrix}\begin{bmatrix} \Lambda & 0\\
0 & \mu
\end{bmatrix}\begin{bmatrix} O_{11}^T & O_{21}^T\\
O_{12}^T & O_{22}
\end{bmatrix}.
\end{equation}
Here we used the fact that $n$ is even, and thus the
$(n+1)\times(n+1)$ real normal matrix $\begin{bmatrix} A & x\\
y^T & z
\end{bmatrix}$ has at least one real eigenvalue.
 The last equality is
equivalent to the following system:
\begin{eqnarray}
A &= & O_{11}\Lambda O_{11}^T+\mu
O_{12}O_{12}^T=O_{11}(\Lambda-\mu
I_n) O_{11}^T+\mu I_n, \label{A_real}\\
x &=& O_{11}\Lambda O_{21}^T+\mu O_{12}O_{22}=O_{11}(\Lambda-\mu
I_n) O_{21}^T, \label{x_real}\\
y^T &=& O_{21}\Lambda O_{11}^T+\mu
{O}_{22}O_{12}^T=O_{21}(\Lambda-\mu I_n) O_{11}^T, \label{y_real}\\
z &= & O_{21}\Lambda O_{21}^T+\mu {O}_{22}^2=O_{21}(\Lambda-\mu
I_n) O_{21}^T+\mu.\label{z_real}
\end{eqnarray}
Setting $C=O_{11}$ and $D=\Lambda-\mu I_n$, we obtain
\eqref{repr1_real} from \eqref{A_real}.

Conversely, if \eqref{repr1_real} holds, we set $O_{11}=C$,
$\Lambda=D+\mu I_n$ and obtain \eqref{A_real}. For $O=\begin{bmatrix} O_{11} & O_{12}\\
O_{21} & O_{22}
\end{bmatrix}$ a minimal orthogonal completion of $C$, we define
$x,y\in\mathbb{R}^n$ and $z\in\mathbb{R}$ by
\eqref{x_real}--\eqref{z_real}.  Then \eqref{diag_real} holds, i.e., the
matrix $\begin{bmatrix} A & x\\
y^T & z
\end{bmatrix}\in\mathbb{R}^{(n+1)\times(n+1)}$ is a real normal
completion of $A$, and thus $\rnd(A)= 1$.

(ii)$\Longleftrightarrow$(iii) If (ii) holds, let
$C=V\diag(1,\ldots,1,t)W^T$ be the SVD of $C$ (here
$V,W\in\mathbb{R}^{n\times n}$ are orthogonal, $0\le t<1$, and
$M=\diag(1,\ldots,1,t)\in\mathbb{R}^{n\times n}$). Clearly,
$N=W^TDW$ is normal, and \eqref{repr2_real} follows.

Conversely, if \eqref{repr2_real} holds, then $N=W^TDW$ with $D$
block diagonal of the form \eqref{bd} and $W$ orthogonal.  For
$C=V\diag(1,\ldots,1,t)W^T$ we have $\ud(C)=1$, and
\eqref{repr1_real} follows.
\end{proof}
\begin{rem}\label{re_not_normal}
\rm{Remark \ref{not_normal} can be restated in the real case as
follows. The matrix $A$ of even size, given by \eqref{repr2_real},
is not normal if and only if in the matrix $N$ partitioned as
$$N=\begin{bmatrix}
N_0 & g\\
h^T & \alpha
\end{bmatrix}$$
(where $\alpha$ is scalar) $g\neq h$ and, in the case where $t\alpha=
0$, also $g\neq -h$. Moreover, if both $M$ and $N$ are invertible then
$A$ is not normal if and only if the standard basis vector $e_n$ is not
an eigenvector of $N^TN^{-1}$. The statement in the last sentence of
Remark \ref{not_normal} is, in general, not valid in the real case. }
\end{rem}
\medskip

\textbf{Open problem:} What is an analogue of Theorem \ref{rnd1}
for the case where $n$ is odd?

\medskip

In the case of even $n$, similarly to the complex case, representation
\eqref{repr1_real} or \eqref{repr2_real} in Theorem \ref{rnd1}, along
with Remark \ref{re_not_normal}, allow one to construct all matrices
$A$ with $\rnd(A)= 1$. However, this does not give a way to check
whether a given real matrix has real normal defect one. A procedure
for this is our further goal.

\subsection{Identification of matrices  with rnd$\mathbf{(A)=1}$ and
construction of their minimal real normal
completions}\label{re_shorter} In the following theorem, which is
the real counterpart of Theorem \ref{thm1}, we establish the
necessary and sufficient conditions for a matrix
$A\in\mathbb{R}^{n\times n}$ to have real normal defect one, and
for any $A$ with $\rnd(A)=1$ we describe all its minimal real
normal completions.

\begin{thm}\label{thm2}
Let $A\in\mathbb{R}^{n\times n}$. Then
\begin{itemize}
    \item[(i)] $\rnd(A)=1$ if and only if $\rank(A^TA-AA^T)=2$ and
  at least one of the equations
    \begin{equation}\label{eq1_r}
(PA^Tu_1-PAu_2)x_1+(PA^Tu_2-PAu_1)x_2=0,
\end{equation}
 \begin{equation}\label{eq2_r}
(PA^Tu_1+PAu_2)x_1+(PA^Tu_2+PAu_1)x_2=0,
\end{equation}
has a solution pair $x_1,x_2\in\mathbb{R}$ satisfying
\begin{equation}\label{cond_r}
x_1^2-x_2^2=d.
\end{equation}\end{itemize}
Here $u_1,u_2\in\mathbb{R}^n$ are the unit eigenvectors of the
matrix $A^TA-AA^T$ corresponding to its nonzero eigenvalues
$\lambda_1=d (>0)$ and $\lambda_2=-d$, and
\begin{equation}\label{P_r}
P=I_n-u_1u_1^T-u_2u_2^T
\end{equation}
 is the orthogonal projection of
$\mathbb{R}^n$ onto $\Null(A^TA-AA^T)$. \begin{itemize}    \item[(ii)]
If $\rnd(A)=1$ then at least one of the following cases occurs:

\textbf{Case 1.} If $x_1$ and $x_2$ satisfy \eqref{eq1_r} and
\eqref{cond_r} then the matrix
\begin{equation}\label{rnc1}
B_1=\begin{bmatrix} A & x_1u_1+x_2u_2\\
{x}_2u_1^T+{x}_1u_2^T & z
\end{bmatrix}
\end{equation}
is a minimal real normal completion of $A$. Here
\begin{equation}\label{z_def_r1}
z=a_{11}-\frac{1}{d}(x_1+x_2)(a_{12}x_1-a_{21}x_2)
\end{equation}
and
\begin{equation}\label{a's_r}
a_{11}=u_1^TAu_1,\quad a_{12}=u_1^TAu_2,\quad a_{21}=u_2^TAu_1;
\end{equation}

\textbf{Case 2.} If $x_1$ and $x_2$ satisfy \eqref{eq2_r} and
\eqref{cond_r} then the matrix
\begin{equation}\label{rnc2}
B_2=\begin{bmatrix} A & x_1u_1+x_2u_2\\
-{x}_2u_1^T-{x}_1u_2^T & z
\end{bmatrix}
\end{equation}
is a minimal real normal completion of $A$. Here
\begin{equation}\label{z_def_r2}
z=a_{11}+\frac{1}{d}(x_1-x_2)(a_{12}x_1+a_{21}x_2)
\end{equation}
and $a_{ij}$'s are defined by \eqref{a's_r}.
\end{itemize}
Any minimal real normal completion of $A$ arises in this way, i.e.,
either as in Case~1 or as in Case~2 above.
\end{thm}
\begin{proof}
Since $\rnd(A)=1$ implies $\nd(A)=1$, it follows from Corollary
\ref{nec_cond} that the rank condition is necessary for $\rnd(A)=1$.
For real $A$ it takes the form $\rank(A^TA-AA^T)=2$. Clearly, it is also
equivalent to
\begin{equation}\label{rank_cond_r}
\det(A^TA-AA^T-\lambda I_n)=(-1)^n\lambda^{n-2}(\lambda^2-d^2),
\end{equation}
with some $d>0$.

Without loss of generality, we can assume that the rank condition is
satisfied. Then we find the unit eigenvectors
$u_1,u_2\in\mathbb{R}^n$ of the matrix $A^TA-AA^T$ corresponding
to its eigenvalues $\lambda_1=d$ and $\lambda_2=-d$.
 Let $u_3$, \ldots, $u_n$ be  an orthonormal basis for $\Null(A^TA-AA^T)$.
Then $U^{\prime}=\begin{bmatrix} u_3 & \ldots, u_n
\end{bmatrix}\in\mathbb{R}^{n\times (n-2)}$ is an isometry, and
\begin{equation}\label{real_U'}
U^\prime U^{\prime T}=P,
\end{equation}
where $P$ is defined in \eqref{P_r}. The matrix $\widetilde{A}=U^TAU$,
where $U=\begin{bmatrix} u_1 & \ldots & u_n\end{bmatrix}$ is
orthogonal, has the form
\begin{equation}\label{A_form_real}
\widetilde{A}=\begin{bmatrix} a_{11} & a_{12} & u\\
a_{21} & a_{11} & v\\
w^T & q^T & S
\end{bmatrix}
\end{equation}
(the identity $a_{11}=u_1^TAu_1=u_2^TAu_2$ is established in the
same way as in Lemma \ref{other_nec}). As in Lemma \ref{xy}, we
obtain that $\rnd(\widetilde{A})=1$ (and therefore, $\rnd(A)=1$)
if and only if there exist $x_1,x_2,y_1,y_2,z\in\mathbb{R}$ such
that the matrix
\begin{equation}\label{real_compl}
\widetilde{B}=\begin{bmatrix}
a_{11} & a_{12} & u & x_1\\
a_{21} & a_{11} & v & x_2\\
w^T & q^T & S & 0\\
y_1 & y_2 & 0 & z
\end{bmatrix}\in\mathbb{R}^{(n+1)\times (n+1)}
\end{equation}
 is normal if and only if there exist
$x_1,x_2,y_1,y_2,z\in\mathbb{R}$ such that
\begin{eqnarray}
(a_{11}-z)x_1+a_{21}x_2 &=&
(a_{11}-z)y_1+a_{12}y_2,\label{real_id1}\\
a_{12}x_1+(a_{11}-z)x_2 &=&
a_{21}y_1+(a_{11}-z)y_2,\label{real_id2}\\
u^Tx_1+v^Tx_2 &=& w^Ty_1+q^Ty_2,\label{real_id3}\\
x_1^2-y_1^2 &=& y_2^2-x_2^2=d,\label{real_id4}\\
x_1x_2 &=& y_1y_2.\label{real_id5}
\end{eqnarray}
It follows from \eqref{real_id4} and \eqref{real_id5} that either
$y_1=x_2$, $y_2=x_1$ or $y_1=-x_2$, $y_2=-x_1$. We will consider
these two cases separately.

\textbf{Case 1: $y_1=x_2$, $y_2=x_1$.} Identities
\eqref{real_id1}--\eqref{real_id4} become
\begin{eqnarray}
(a_{11}-z)x_1+a_{21}x_2 &=&
(a_{11}-z)x_2+a_{12}x_1,\label{real_id1'}\\
a_{12}x_1+(a_{11}-z)x_2 &=&
a_{21}x_2+(a_{11}-z)x_1,\label{real_id2'}\\
(u^T-q^T)x_1+(v^T-w^T)x_2 &=& 0,\label{real_id3'}\\
x_1^2-x_2^2 &=& d.\label{real_id4'}
\end{eqnarray}
Clearly, \eqref{real_id1'} and \eqref{real_id2'} are equivalent,
and it follows from \eqref{real_id1'} and \eqref{real_id4'} that
\begin{equation*}
z=a_{11}-\frac{a_{12}x_1-a_{21}x_2}{x_1-x_2}=a_{11}-\frac{1}{d}(x_1+x_2)(a_{12}x_1-a_{21}x_2)
\end{equation*}
(cf. \eqref{z_def}). Next, it follows from \eqref{A_form_real}
that
\begin{equation*}
u^T=U^{\prime T}A^Tu_1,\quad v^T=U^{\prime T}A^Tu_2,\quad
w^T=U^{\prime T}Au_1,\quad  q^T=U^{\prime T}Au_2.
\end{equation*}
Multiplying both parts of these equalities on the left by
$U^\prime$  and taking into account \eqref{real_U'}, we obtain
vectors
\begin{eqnarray}
\mathbf{u}^T =U^\prime u^T=PA^Tu_1, & \mathbf{v}^T=U^\prime
v^T=PA^Tu_2,\label{hat1}\\
\mathbf{w}^T=U^\prime w^T=PAu_1, & \mathbf{q}^T=U^\prime
q^T=PAu_2.\label{hat2}
\end{eqnarray}
Since $U^\prime$ is an isometry, \eqref{real_id3'} is equivalent
to
\begin{equation}\label{real_main}
(\mathbf{u}^T-\mathbf{q}^T)x_1+(\mathbf{v}^T-\mathbf{w}^T)x_2 = 0.
\end{equation}
Note that the definition of vectors
$\mathbf{u}^T,\mathbf{v}^T,\mathbf{w}^T,\mathbf{q}^T\in\range(P)\subset\mathbb{R}^n$
in \eqref{hat1}--\eqref{hat2} is independent of $U^\prime$, i.e.,
of the choice of basis vectors $u_3$, \ldots, $u_n$ in $\range
(P)=\Null(A^TA-AA^T)$.

\textbf{Case 2: $y_1=-x_2$, $y_2=-x_1$.} Identities
\eqref{real_id1}--\eqref{real_id4} become
\begin{eqnarray}
(a_{11}-z)x_1+a_{21}x_2 &=&
-(a_{11}-z)x_2-a_{12}x_1,\label{real_id1''}\\
a_{12}x_1+(a_{11}-z)x_2 &=&
-a_{21}x_2-(a_{11}-z)x_1,\label{real_id2''}\\
(u^T+q^T)x_1+(v^T+w^T)x_2 &=& 0,\label{real_id3''}\\
x_1^2-x_2^2 &=& d.\label{real_id4''}
\end{eqnarray}
Clearly, \eqref{real_id1''} and \eqref{real_id2''} are equivalent,
and it follows from \eqref{real_id1''} and \eqref{real_id4''} that
\begin{equation*}
z=a_{11}+\frac{a_{12}x_1+a_{21}x_2}{x_1+x_2}=a_{11}+\frac{1}{d}(x_1-x_2)(a_{12}x_1+a_{21}x_2).
\end{equation*}
 Next, we obtain vectors $\mathbf{u}^T,\mathbf{v}^T,\mathbf{w}^T,\mathbf{q}^T\in\range(P)\subset\mathbb{R}^n$
 as in Case 1.
Since $U^\prime$ is an isometry, \eqref{real_id3''} is equivalent
to
\begin{equation}\label{real_main''}
(\mathbf{u}^T+\mathbf{q}^T)x_1+(\mathbf{v}^T+\mathbf{w}^T)x_2 = 0.
\end{equation}

It follows from the analysis of cases above that $\rnd(A)=1$ if
and only if at least one of the equations \eqref{real_main} and
\eqref{real_main''} has a solution pair $x_1,x_2\in\mathbb{R}$
satisfying \eqref{cond_r}, which proves part (i).

If $x_1,x_2$ satisfy \eqref{real_main} (resp.,
\eqref{real_main''}) and \eqref{cond_r} then $y_1=x_2$, $y_2=x_1$
and $z$ defined by \eqref{z_def_r1} (resp., $y_1=-x_2$, $y_2=-x_1$
and $z$ defined by \eqref{z_def_r2}) determine the minimal real
normal completion $\widetilde{B}$ of the matrix $\widetilde{A}$ by
\eqref{real_compl}. Then the matrix
\begin{equation*}
B =\begin{bmatrix}
U & 0\\
0  & 1
\end{bmatrix}\begin{bmatrix}
a_{11} & a_{12} & u & x_1\\
a_{21} & a_{11} & v & x_2\\
w^T & q^T & S & 0\\
y_1 & y_2 & 0 & z
\end{bmatrix}\begin{bmatrix}
U^T & 0\\
0 & 1
\end{bmatrix}
= \begin{bmatrix}
A & u_1x_1+u_2x_2\\
{y}_1u_1^T+{y}_2u_2^T & z\\
\end{bmatrix}
\end{equation*}
 is a minimal real normal completion of $A$. Since $B=B_1$ in Case
 1 and $B=B_2$ in Case 2, this proves part (ii) of the theorem.
\end{proof}
\begin{cor}\label{nd_rnd}
For a matrix $A\in\mathbb{R}^{n\times n}$, $\rnd(A)=1$ if and only
if $\nd(A)=1$.
\end{cor}
\begin{proof}
Since we have $\nd(A)\le\rnd(A)$, it suffices to prove that if
$\nd(A)=1$ then $\rnd(A)=1$. Suppose that $\nd(A)=1$. Then, as
described in Section \ref{shorter},  equation \eqref{short} has a
solution $x=\col(x_{R1},x_{I1},x_{R2},x_{I2})\in\mathbb{R}^4$ with
\begin{equation}\label{ineq}
x_{R1}^2+x_{I1}^2>x_{R2}^2+x_{I2}^2
\end{equation}
 (see Theorem \ref{thm1}).
 The matrix
$\mathbf{Q}$ in this (real) case has the form
 $$\mathbf{Q}=\begin{bmatrix}
\mathbf{u}^T-\mathbf{q}^T & 0 & \mathbf{v}^T-\mathbf{w}^T
& 0\\
0 & \mathbf{u}^T+\mathbf{q}^T & 0 & \mathbf{v}^T+\mathbf{w}^T
\end{bmatrix}.$$
Thus, in this case \eqref{short} is equivalent to the pair of equations
\begin{eqnarray*}
(\mathbf{u}^T-\mathbf{q}^T)x_{R1}+(\mathbf{v}^T-\mathbf{w}^T)x_{R2} &=& 0,\\
(\mathbf{u}^T+\mathbf{q}^T)x_{I1}+(\mathbf{v}^T+\mathbf{w}^T)x_{I2}
&=& 0.
\end{eqnarray*}
Since in \eqref{ineq} either $x_{R1}^2>x_{R2}^2$ or
$x_{I1}^2>x_{I2}^2$, at least one of the equations
\eqref{real_main} or \eqref{real_main''} (or equivalently, at
least one of the equations \eqref{eq1_r} and \eqref{eq2_r}) has a
solution pair $x_1,x_2\in\mathbb{R}$ with $x_1^2>x_2^2$ (and thus,
a solution pair $x_1,x_2\in\mathbb{R}$ satisfying
$x_1^2-x_2^2=d$), which by Theorem \ref{thm2} means that
$\rnd(A)=1$.
\end{proof}
\medskip

\textbf{Open problem:} Is it true that for any
$A\in\mathbb{R}^{n\times n}$ one has $\rnd(A)=\nd(A)$?

\medskip

As in the complex case, we will describe now a procedure (in this
setting based on Theorem \ref{thm2}) which allows one to determine
whether $\rnd(A)=1$ for a given matrix $A\in\mathbb{R}^{n\times n}$.
Moreover, this
procedure allows one to find all such solutions, and then, applying
part (ii) of Theorem \ref{thm2}, all minimal normal completions of $A$
when $\rnd(A)=1$.
\medskip

\begin{center}{\textbf{The procedure}}
\end{center}

\medskip

\noindent\textbf{Begin}
\medskip

\textbf{Step 1.} Check the rank condition. If it
holds --- go to Step 2. Otherwise, stop: $\rnd(A)>1$.

\textbf{Step 2.} Write \eqref{eq1_r} in the form \eqref{real_main},
where $\mathbf{u}$, $\mathbf{v}$, $\mathbf{w}$, $\mathbf{q}$ are
defined by \eqref{hat1} and \eqref{hat2} (see
 Theorem \ref{thm2} for the definition of $P$, $u_1$, and $u_2$). Find
$m_1=\rank(\mathbf{u}^T-\mathbf{q}^T,\mathbf{v}^T-\mathbf{w}^T)$.

\textbf{Step 3.} Depending on $m_1$, consider the following cases.

\begin{itemize}
\item[\textbf{(1a)}] $m_1=0$. In this case, any
$x_1,x_2\in\mathbb{R}$ with $x_1^2-x_2^2=d$ solve
\eqref{real_main}.

\item[\textbf{(1b)}] $m_1=1$, i.e.,
    $\mathbf{u}^T-\mathbf{q}^T=\alpha b$,
    $\mathbf{v}^T-\mathbf{w}^T=\beta b$, with some nonzero vector
    $b\in\range(P)$ and  $\alpha,\beta\in\mathbb{R}$, and
    additionally $|\alpha|\ge|\beta|$. In this case, \eqref{real_main} is
    equivalent to $\alpha x_1+\beta x_2=0$, and has no solutions
    satisfying \eqref{cond_r}.

\item[\textbf{(1c)}] $m_1=1$, i.e.,
    $\mathbf{u}^T-\mathbf{q}^T=\alpha b$,
    $\mathbf{v}^T-\mathbf{w}^T=\beta b$, with some nonzero vector
    $b\in\range(P)$ and  $\alpha,\beta\in\mathbb{R}$, and
    additionally $|\alpha|<|\beta|$. In this case, \eqref{real_main} is
    equivalent to $\alpha x_1+\beta x_2=0$, and has the solutions
\begin{equation*}
x_1=\pm\beta\sqrt{\frac{d}{\beta^2-\alpha^2}},\quad
x_2=\mp\alpha\sqrt{\frac{d}{\beta^2-\alpha^2}}
\end{equation*}
satisfying \eqref{cond_r}.

\item[\textbf{(1d)}] $m_1=2$. In this case, \eqref{real_main} has no
    nonzero solutions, and thus has no solutions satisfying
    \eqref{cond_r}.
\end{itemize}

\textbf{Step 4.} Write \eqref{eq2_r} in the form \eqref{real_main''},
where $\mathbf{u}$, $\mathbf{v}$, $\mathbf{w}$, $\mathbf{q}$ are
defined by \eqref{hat1} and \eqref{hat2} (see  Theorem \ref{thm2} for
the definition of $P$, $u_1$, and $u_2$). Find
$m_2=\rank(\mathbf{u}^T+\mathbf{q}^T,\mathbf{v}^T+\mathbf{w}^T)$.

\textbf{Step 5.} Depending on $m_2$, consider the following cases.

\begin{itemize}

\item[\textbf{(2a)}] $m_2=0$. In this case, any
$x_1,x_2\in\mathbb{R}$ with $x_1^2-x_2^2=d$ solve
\eqref{real_main''}.

\item[\textbf{(2b)}] $m_2=1$, i.e.,
    $\mathbf{u}^T+\mathbf{q}^T=\gamma h$,
    $\mathbf{v}^T+\mathbf{w}^T=\delta h$, with some nonzero
    vector $h\in\range(P)$ and  $\gamma,\delta\in\mathbb{R}$, and
    additionally $|\gamma|\ge|\delta|$. In this case,
    \eqref{real_main''} is equivalent to $\gamma x_1+\delta x_2=0$,
    and has no solutions satisfying \eqref{cond_r}.

\item[\textbf{(2c)}] $m_2=1$, i.e.,
    $\mathbf{u}^T+\mathbf{q}^T=\gamma h$,
    $\mathbf{v}^T+\mathbf{w}^T=\delta h$, with some nonzero
    vector $h\in\range(P)$ and  $\gamma,\delta\in\mathbb{R}$, and
    additionally $|\gamma|<|\delta|$. In this case, \eqref{real_main''}
    is equivalent to $\gamma x_1+\delta x_2=0$, and has the
    solutions
\begin{equation*}
x_1=\pm\delta\sqrt{\frac{d}{\delta^2-\gamma^2}},\quad
x_2=\mp\gamma\sqrt{\frac{d}{\delta^2-\gamma^2}}
\end{equation*}
satisfying \eqref{cond_r}.

\item[\textbf{(2d)}] $m_2=2$. In this case, \eqref{real_main''}
has no nonzero solutions, and thus, has no solutions satisfying
\eqref{cond_r}.
\end{itemize}

\textbf{Step 6.} $\rnd(A)=1$ if and only if neither of the combinations
(1b)\&(2b), (1b)\&(2d), (1d)\&(2b), (1d)\&(2d) occur. If it does, stop:
$\rnd(A)>1$. Otherwise, for each pair $x_1,x_2\in\mathbb{R}$
obtained at Step 3, find a minimal real normal completion of $A$ as
described in \eqref{rnc1}--\eqref{a's_r}; for each pair
$x_1,x_2\in\mathbb{R}$ obtained at Step 4, find a minimal real normal
completion of $A$ as described in \eqref{a's_r}--\eqref{z_def_r2}.

\medskip
\noindent \textbf{End}

\medskip

Of course, if one is interested only in checking whether $\rnd(A)=1$,
the procedure can be terminated as soon as any of cases (1a), (1c),
(2a), (2c) occurs.

\begin{ex}\label{real_ex:a}
\rm{In Example \ref{ex:a},
$$A=\begin{bmatrix} 1 & 0 & 0\\
0 & 1 & 1\\
1 & 0 & 1
\end{bmatrix}$$ is a matrix with real entries,
and $$ A^TA-AA^T=\begin{bmatrix} 1 & 0 & 0\\
0 & -1 & 0\\
0 & 0 & 0
\end{bmatrix},$$
so that the rank condition is satisfied. By Corollaries \ref{n=3} and
\ref{nd_rnd}, $\rnd(A)=1$. We have $u_1=e_1$,
$u_2=e_2$, and $$P=I-u_1u_1^T-u_2u_2^T=\begin{bmatrix} 0 & 0 & 0\\
0 & 0 & 0\\
0 & 0 & 1
\end{bmatrix}.$$
 Then, in the procedure above,
$\mathbf{u}^T=\mathbf{q}^T=0$, $\mathbf{v}^T=\mathbf{w}^T=e_3$.
Since
$m_1=\rank(\mathbf{u}^T-\mathbf{q}^T,\mathbf{v}^T-\mathbf{w}^T)=0$,
as in Case (1a),  any $x_1,x_2\in\mathbb{R}$ with $x_1^2-x_2^2=1$
solve \eqref{real_main}. We have  $y_1=x_2$ and $y_2=x_1$.
According to \eqref{z_def_r1},  $z=1$. Therefore, for any
$x_1\in\mathbb{R}$ such that $|x_1|\ge 1$,
$$\begin{bmatrix} 1 & 0 & 0 & x_1\\
0 & 1 & 1 & \pm\sqrt{x_1^2-1}\\
1 & 0 & 1 & 0\\
\pm\sqrt{x_1^2-1} & x_1 & 0 & 1
\end{bmatrix}$$
is a minimal real normal completion of $A$. We also have
$$m_2=\rank(\mathbf{u}^T+\mathbf{q}^T,\mathbf{v}^T+\mathbf{w}^T)=\rank(0,2e_3)=1,$$ and as in Case (2c), $h=e_3$,
$\gamma=0$, $\delta=2$, so  that  $x_1=\pm 1=-y_2,x_2=0=-y_1$.
According to \eqref{z_def_r2},  $z=1$. Thus, $$\begin{bmatrix} 1 & 0 & 0 & \pm 1\\
0 & 1 & 1 & 0\\
1 & 0 & 1 & 0\\
0 & \mp 1 & 0 & 1
\end{bmatrix}$$
is a minimal real normal completion of $A$. We therefore obtain the
set of minimal real normal completions of $A$ arising from Cases (1a)
and (2c). Note that the minimal real normal completions of $A$ in this
example are special cases of minimal normal completions of $A$ as in
\eqref{ncompl}, where $x_1\in\mathbb{R}\colon |x_1|\ge 1$,
$x_2=\pm\sqrt{x_1^2-1}$, and $\mu=1$, or where $x_1=i$, $x_2=0$,
and $\mu=\pm i$.

We know from Example \ref{ex:a} that $\ud(A)=2$ and that $A$ is
unitarily (and therefore  orthogonally) irreducible. This example shows
that $\rnd(A)<\ud(A)$ is possible for an orthogonally irreducible
matrix $A$.}
\end{ex}
\begin{ex}\label{shift_real}
\rm{As we saw in Example \ref{shift}, $\nd(A)=\ud(A)=1$, for the shift
matrices $A$ of size greater than $3$. Moreover, in this case all
minimal normal completions are actually minimal unitary
completions. Since the shift matrix $A$ has real entries, and
$$\nd(A)\le\rnd(A)\le\ud(A),$$ we conclude that $\rnd(A)=\ud(A)=1$,
and all minimal real normal completions are actually minimal
orthogonal completions. This corresponds to $\zeta$ and $\rho$ in
\eqref{uc_shift} independently taking values $1$ or $-1$.}
\end{ex}

\subsection{The generic case}\label{re_finer}
We will describe now the generic situation in each matrix dimension
$n$. As in the complex case, a finer analysis is needed for this.
However, in the real case our analysis is more straightforward and
does not use a ``heavy machinery" of symmetric extensions.

 For a real matrix $A$ satisfying the rank condition, it follows from Lemma
\ref{other_nec} that
 the following identities hold:
\begin{equation}\label{real_rel}
\mathbf{u}\mathbf{u}^T=\mathbf{q}\mathbf{q}^T,\quad
\mathbf{v}\mathbf{v}^T=\mathbf{w}\mathbf{w}^T,\quad
\mathbf{u}\mathbf{v}^T=\mathbf{w}\mathbf{q}^T,\quad
\mathbf{u}\mathbf{w}^T=\mathbf{v}\mathbf{q}^T.
\end{equation}
Consequently,
\begin{equation*}
(\mathbf{u}+\mathbf{q})(\mathbf{u}-\mathbf{q})^T=0,\
(\mathbf{v}+\mathbf{w})(\mathbf{v}-\mathbf{w})^T=0,\
(\mathbf{v}+\mathbf{w})(\mathbf{u}-\mathbf{q})^T=0,\
(\mathbf{u}+\mathbf{q})(\mathbf{v}-\mathbf{w})^T=0,
\end{equation*}
i.e., each of the vectors $(\mathbf{u}+\mathbf{q})^T$ and
$(\mathbf{v}+\mathbf{w})^T$ is orthogonal to each of the vectors
$(\mathbf{u}-\mathbf{q})^T$ and $(\mathbf{v}-\mathbf{w})^T$. Note
that the vectors $\mathbf{u}^T$, $\mathbf{v}^T$, $\mathbf{w}^T$,
and $\mathbf{q}^T$  belong to $\range(P)$ whose dimension is
$n-2$.

Restricting our attention to real matrices in ${\frak M}_n$, we now
consider different values of $n$ separately. 

 \begin{description}
    \item[$\mathbf{n=2}$] In this case, vectors
        $\mathbf{u}^T,\mathbf{v}^T,\mathbf{w}^T$, and
        $\mathbf{q}^T$ do not arise, thus necessarily $\rnd(A)=1$. This
        follows also from the fact the $\nd(A)=1$ by
        Corollary~\ref{nd_rnd}.

 \item[$\mathbf{n=3}$] In this case, vectors
     $\mathbf{u}^T,\mathbf{v}^T,\mathbf{w}^T$, and $\mathbf{q}^T$
     are collinear, and in view of \eqref{real_rel} either
     $\mathbf{u}^T=\mathbf{q}^T$ and
     $\mathbf{v}^T=\mathbf{w}^T$, or $\mathbf{u}^T=-\mathbf{q}^T$
     and $\mathbf{v}^T=-\mathbf{w}^T$. Thus either Case (1a) or
     Case (2a) in the Procedure occurs. Therefore, necessarily
     $\rnd(A)=1$ (this follows also from Corollaries \ref{n=3} and
     \ref{nd_rnd}).

 \item[$\mathbf{n=4}$] Generically, $\mathbf{u}^T\neq\pm\mathbf{q}^T$, and
$\mathbf{v}^T\neq\pm\mathbf{w}^T$. Since $\dim(\range(P))=2$, the
vectors $(\mathbf{u}+\mathbf{q})^T$ and
$(\mathbf{v}+\mathbf{w})^T$ are collinear and orthogonal to
$(\mathbf{u}-\mathbf{q})^T$ and $(\mathbf{v}-\mathbf{w})^T$, which
are also collinear. Then both the combination of Case (1b) and
Case (2b), and the combination of Case (1c) and Case (2c) (and
thus, both $\rnd(A)=1$ and $\rnd(A)>1$) occur on the sets whose
interior is nonempty in the relative topology of the manifold
${\frak M}_4$. Indeed, the first combination occurs when we fix
$\mathbf{u}^T,\mathbf{q}^T$ and make
$\mathbf{v}\mathbf{v}^T=\mathbf{w}\mathbf{w}^T$ small enough, and
the second combination occurs when we fix
$\mathbf{v}^T,\mathbf{w}^T$ and make
$\mathbf{u}\mathbf{u}^T=\mathbf{q}\mathbf{q}^T$  small enough.

 \item[$\mathbf{n=5}$] Generically,
     $\mathbf{u}^T\neq\pm\mathbf{q}^T$, and
     $\mathbf{v}^T\neq\pm\mathbf{w}^T$. Since
     $\dim(\range(P))=3$, at least one of the pairs of vectors
     (generically, only one such pair), $(\mathbf{u}+\mathbf{q})^T$
     and $(\mathbf{v}+\mathbf{w})^T$ or
     $(\mathbf{u}-\mathbf{q})^T$ and $(\mathbf{v}-\mathbf{w})^T$, is
     collinear. As in the case $n=4$, for a collinear pair, both cases (b)
     and (c) occur on the sets whose interior is nonempty in the
     relative topology of ${\frak M}_5$. Thus, combinations of Case (1b) and
     Case (2d), Case (1d) and Case (2b) (where $\rnd(A)>1$), and
     combinations of Case (1c) and Case (2d), Case (1d) and Case
     (2c) (where $\rnd(A)=1$) occur on the sets whose interior is
     nonempty in the relative topology of ${\frak M}_5$. 

 \item[$\mathbf{n\ge 6}$] Since $\dim(\range(P))\ge 4$, the pairs
     $(\mathbf{u}+\mathbf{q})^T$, $(\mathbf{v}+\mathbf{w})^T$ and
     $(\mathbf{u}-\mathbf{q})^T$, $(\mathbf{v}-\mathbf{w})^T$ are
     generically linearly independent. Therefore, the combination of
     Case (1d) and Case (2d) (corresponding to $\rnd(A)>1$) occurs
     generically.
\end{description}

Thus, we see that the generic situation in the real case is
similar to the one in the complex case.

\section{Commuting completion
problems}\label{sec:commut}

The problem of finding commuting completions of a $N$-tuple of
$n\times n$ matrices was raised in \cite{DST}, where a special
emphasis was placed on symmetric completions of $N$-tuples of
symmetric matrices. In \cite{KW}, an inverse completion $(A_{\rm
ext},B_{\rm ext})$ of a pair $(A,B)$ was constructed. Namely, $A_{\rm
ext}$, $B_{\rm ext}$ by definition satisfy $A_{\rm ext}B_{\rm
ext}=\alpha I$ with a non-zero scalar $\alpha$, and therefore
commute. Our results from Sections \ref{sec:compl} and \ref{sec:re}
can be used to tackle commuting completion problems in the classes
of Hermitian (resp., symmetric, or symmetric/antisymmetric) pairs of
matrices.

\subsection{The commuting Hermitian completion problem.}\label{herm}
Let $(A_1, A_2)$ be a pair of Hermitian matrices of size $n\times
n$. We define the \emph{commuting Hermitian defect of $A_1$ and
$A_2$}, denoted $\chd(A_1,A_2)$, as the minimal possible $p$ such
that there
exist commuting Hermitian matrices $B_1=\begin{bmatrix} A_1 & *\\
* & * \end{bmatrix}$ and $B_2=\begin{bmatrix} A_2 & *\\
* & * \end{bmatrix}$ of size $(n+p)\times(n+p)$.
We call such a pair $(B_1,B_2)$  of size $(n+\chd(A_1,A_2))\times
(n+\chd(A_1,A_2))$ a \emph{minimal commuting Hermitian completion
of $(A_1,A_2)$}.

Since $(B_1,B_2)$ is a commuting Hermitian completion of a pair
$(A_1,A_2)$ of Hermitian matrices if and only if $B=B_1+iB_2$ is a
normal completion of $A=A_1+iA_2$, and therefore
$\chd(A_1,A_2)=\nd(A_1+iA_2)$, the results from Section \ref{shorter}
allow one to check whether $\chd(A_1,A_2)=1$, and when this is the
case --- to construct all minimal commuting Hermitian completions of
$(A_1, A_2)$. For example, Theorem \ref{thm1} yields the following.

\begin{thm}\label{thm:chd}
Let $A_1,A_2\in\mathbb{C}^{n\times n}$ be Hermitian.
\begin{itemize}
    \item[(i)] $\chd(A_1,A_2)=1$ if and only if $\rank(A_1A_2-A_2A_1)=2$ and
    the equation
    \begin{equation}\label{eq_com}
PA_1(t_1u_1-\overline{t}_1u_2)=iPA_2(t_2u_1+\overline{t}_2u_2)
\end{equation}
has a solution pair $t_1,t_2\in\mathbb{C}$ satisfying
\begin{equation}\label{cond_com}
\Re(\overline{t}_1t_2)=d.
\end{equation}\end{itemize}
Here $u_1,u_2\in\mathbb{C}^n$ are the unit eigenvectors of the
matrix $2i(A_1A_2-A_2A_1)$ corresponding to its nonzero eigenvalues
$\lambda_1=d (>0)$ and $\lambda_2=-d$, and
$P=I_n-u_1u_1^*-u_2u_2^*$ is the orthogonal projection of
$\mathbb{C}^n$ onto $\Null(A_1A_2-A_2A_1)$. \begin{itemize}
\item[(ii)] If $\chd(A_1,A_2)=1$, $t_1$ and $t_2$ satisfy \eqref{eq_com}
    and \eqref{cond_com}, and $\mu\in\mathbb{T}$ is
    arbitrary then the pair $(B_1,B_2)$ of matrices
\begin{equation}\label{ch_compl_1}
B_1=\begin{bmatrix} A_1 & \frac{\mu}{2}(t_2u_1+\overline{t}_2u_2)\\
\frac{\overline{\mu}}{2}(\overline{t}_2u_1^*+{t}_2u_2^*) & z_1
\end{bmatrix},
\end{equation}
\begin{equation}\label{ch_compl_2}
 B_2=\begin{bmatrix} A_2 & \frac{\mu}{2i}(t_1u_1-\overline{t}_1u_2)\\
-\frac{\overline{\mu}}{2i}(\overline{t}_1u_1^*-{t}_1u_2^*) & z_2
\end{bmatrix}
\end{equation}
is a minimal commuting Hermitian completion of $(A_1,A_2)$. Here
\begin{equation}\label{z_1}
z_1=u_1^*A_1u_1-\frac{1}{d}\left(\Im(t_2^2u_2^*A_2u_1)+\Re(t_1t_2u_2^*A_1u_1)\right)
\end{equation}
and
\begin{equation}\label{z_2}
z_2=u_1^*A_2u_1-\frac{1}{d}\left(\Im(t_1^2u_2^*A_1u_1)-\Re(t_1t_2u_2^*A_2u_1)\right).
\end{equation}\end{itemize}
All minimal commuting Hermitian completions of $(A_1,A_2)$ arise in
this way. \end{thm}

\begin{proof}
Letting $A=A_1+iA_2$, we observe that
$A^*A-AA^*=2i(A_1A_2-A_2A_1)$. It is straightforward to verify
that, under the change of variables
 $t_1=x_1-\overline{x}_2$, $t_2=x_1+\overline{x}_2$,
 condition \eqref{eq} in Theorem \ref{thm1} is equivalent to condition
 \eqref{eq_com}, \eqref{cond} is equivalent to
 \eqref{cond_com}, $B_1$ and $B_2$ defined by \eqref{ch_compl_1} and \eqref{ch_compl_2} are
 Hermitian and such that
 $B=B_1+iB_2$ is as in \eqref{ncompl}, $z_1$ and $z_2$ defined by \eqref{z_1} and \eqref{z_2}
 are real and such that
$z=z_1+iz_2$ is as in \eqref{z_def}. Thus, this theorem is
equivalent  to Theorem \ref{thm1}.
\end{proof}

\subsection{The commuting completion problem in the class of pairs of symmetric and antisymmetric
matrices.}\label{sym-antisym} Let $A_1=A_1^T\in\mathbb{R}^{n\times
n}$ and $A_2=-A_2^T\in\mathbb{R}^{n\times n}$. It is natural to
ask what is the minimal $p$ such that there exist commuting
matrices
$B_1=B_1^T=\begin{bmatrix} A_1 & *\\
* & * \end{bmatrix} \in\mathbb{R}^{(n+p)\times
(n+p)}$ and $B_2=-B_2^T=\begin{bmatrix} A_2 & *\\
* & * \end{bmatrix}\in\mathbb{R}^{(n+p)\times (n+p)}$. Such a pair
$(B_1,B_2)$ is a \emph{minimal commuting completion of $(A_1,A_2)$
in the class of pairs of symmetric and antisymmetric matrices}.
Since $(B_1,B_2)$ is a commuting completion of $(A_1,A_2)$ in this
class if and only if $B=B_1+B_2$ is a real normal completion of
$A=A_1+A_2$, our results from Section \ref{sec:re} can be restated
in terms of pairs of matrices in this class. We omit the details,
since the reasoning is similar to the one in Section \ref{herm}.

\subsection{The commuting symmetric completion
problem.}\label{sym}

In this section, we consider the commuting completion problem in
the class of pairs of symmetric matrices. This is a special case
of the problem raised in Degani et al. \cite{DST} (see the first
paragraph of Section \ref{sec:commut}) for $N=2$. The authors of
\cite{DST} presented an approach to $n$-dimensional cubature
formulae where the cubature nodes are obtained by means of
commuting completions of certain matrix tuples. While their
commuting completion problem is stated in a certain subclass of
tuples of symmetric matrices, some observations were also made for
the problem in the whole class. In particular, the question on the
minimal possible size of completed matrices was accentuated as
important.

Let $A_1=A_1^T\in\mathbb{R}^{n\times n}$ and
$A_2=A_2^T\in\mathbb{R}^{n\times n}$. We define the
\emph{commuting symmetric defect of $A_1$ and $A_2$}, denoted
$\csd(A_1,A_2)$, as the minimal possible $p$ such that there exist
commuting
symmetric matrices $B_1=\begin{bmatrix} A_1 & *\\
* & * \end{bmatrix},B_2=\begin{bmatrix} A_2 & *\\
* & * \end{bmatrix}\in\mathbb{R}^{(n+p)\times(n+p)}$.
Such a pair $(B_1,B_2)$ of size $(n+\csd(A_1,A_2))\times
(n+\csd(A_1,A_2))$ is a \emph{minimal commuting symmetric
completion of the pair $(A_1,A_2)$}.

We note that $(B_1,B_2)$ is a commuting symmetric completion of a
pair $(A_1,A_2)$ of real symmetric matrices if and only if
$B=B_1+iB_2$ is a normal, and simultaneously complex symmetric,
completion of $A=A_1+iA_2$.  We also observe that  a priori
\begin{equation}\label{csd-chd}
\csd(A_1,A_2)\ge\chd(A_1,A_2).
\end{equation}

\textbf{Open problem.} Is it true that for any pair $(A_1,A_2)$ of
real symmetric matrices one has $\csd(A_1,A_2)=\chd(A_1,A_2)$?

\medskip

This problem is somewhat similar to the open problem stated in
Section \ref{re_shorter}. The latter actually asks whether a
minimal normal completion of a real matrix can be chosen to be
real, while the former is equivalent to the question whether a
minimal normal completion of a complex symmetric matrix can be
chosen to be complex symmetric.

 The following theorem shows that,
for a pair $(A_1,A_2)$ of real symmetric matrices,
\begin{equation*}\label{csd1}
\csd(A_1,A_2)=1\Longleftrightarrow\chd(A_1,A_2)=1,
\end{equation*}
which motivates the open problem above. Moreover, this theorem
actually shows that if $\csd(A_1,A_2)=1$ then the set of all minimal
commuting symmetric completions $(B_1,B_2)$ of
 $(A_1,A_2)$ can be obtained by putting in Theorem \ref{thm:chd}
$u_2=\overline{u}_1$ and $\mu=1$.

\begin{thm}\label{thm:csd}
Let $A_1,A_2\in\mathbb{R}^{n\times n}$ be symmetric.
\begin{itemize}
    \item[(i)] $\csd(A_1,A_2)=1$ if and only if $\rank(A_1A_2-A_2A_1)=2$ and
    the equation
    \begin{equation}\label{eq_sym}
PA_1\Im(t_1u_1)=PA_2\Re(t_2u_1)
\end{equation}
has a solution pair $t_1,t_2\in\mathbb{C}$ satisfying
\begin{equation}\label{cond_sym}
\Re(\overline{t}_1t_2)=d.
\end{equation}
Here $u_1\in\mathbb{C}^n$ is the unit eigenvector of the matrix
$2i(A_1A_2-A_2A_1)$ corresponding to its eigenvalue
$\lambda_1=d(>0)$, and $P=I_n-u_1u_1^*-\overline{u}_1u_1^T$.
    \item[(ii)]
If $\csd(A_1,A_2)=1$, $t_1$ and $t_2$ satisfy \eqref{eq_sym} and
\eqref{cond_sym} then the pair $(B_1,B_2)$ of matrices
\begin{equation}\label{compl_sym_1}
B_1=\begin{bmatrix} A_1 & \Re(t_2u_1)\\
\Re({t}_2u_1)^T & z_1
\end{bmatrix},
\end{equation}
\begin{equation}\label{compl_sym_2}
 B_2=\begin{bmatrix} A_2 & \Im(t_1u_1)\\
\Im(t_1u_1)^T & z_2
\end{bmatrix}
\end{equation}
is a minimal commuting symmetric completion of $(A_1,A_2)$. Here
\begin{equation}\label{z_1_sym}
z_1=u_1^*A_1u_1-\frac{1}{d}\left(\Im(t_2^2u_1^TA_2u_1)+\Re(t_1t_2u_1^TA_1u_1)\right)
\end{equation}
and
\begin{equation}\label{z_2_sym}
z_2=u_1^*A_2u_1-\frac{1}{d}\left(\Im(t_1^2u_1^TA_1u_1)-\Re(t_1t_2u_1^TA_2u_1)\right).
\end{equation}
\end{itemize}

All minimal commuting symmetric completions of $(A_1,A_2)$ arise
in this way.
\end{thm}

\begin{proof} (i)
By \eqref{csd-chd}, if $\csd(A_1,A_2)=1$ then $\chd(A_1,A_2)=1$.
Therefore, by Theorem \ref{thm:chd}, $\rank(A_1A_2-A_2A_1)=2$ and
equation \eqref{eq_com} has a solution pair $t_1,t_2\in\mathbb{C}$
satisfying \eqref{cond_com}. If $u_1$ is the unit eigenvector of
the Hermitian matrix $2i(A_1A_2-A_2A_1)$ corresponding to its
eigenvalue $\lambda_1=d(>0)$ then $\overline{u}_1$ is the unit
eigenvector corresponding to the eigenvalue $\lambda_2=-d$. Thus,
we can choose in Theorem \ref{thm:chd} $u_2=\overline{u}_1$. Then
$P=I_n-u_1u_1^*-\overline{u}_1u_1^T$ is a real $n\times n$ matrix,
and equation \eqref{eq_com} becomes \eqref{eq_sym}.

Conversely, if $\rank(A_1A_2-A_2A_1)=2$ and equation
\eqref{eq_sym} (which is equivalent to \eqref{eq_com} in our case)
has a solution pair $t_1,t_2\in\mathbb{C}$ satisfying
\eqref{cond_sym} (= \eqref{cond_com}) then by Theorem
\ref{thm:chd}, $\chd(A_1,A_2)=1$. For any such $t_1,t_2$ the
corresponding minimal commuting Hermitian completions
$(B_1,B_2)$ of $(A_1,A_2)$ have the form
\eqref{ch_compl_1}--\eqref{ch_compl_2}. We observe that since
$u_2=\overline{u}_1$, the matrices $B_1$ and $B_2$ are real
symmetric if and only if $\mu=1$ or $\mu=-1$. Consequently,
$\csd(A_1,A_2)=1$.

(ii) If $t_1,t_2\in\mathbb{C}$ satisfy \eqref{eq_sym}--\eqref{cond_sym}
then so do
  $t'_1=-t_1$ and $t'_2=-t_2$. Therefore, we do not
miss any minimal commuting symmetric completions of $(A_1,A_2)$ if
in \eqref{ch_compl_1}--\eqref{ch_compl_2} we choose  $t_1,t_2$ as
above and fix $\mu=1$. Finally, since $(B_1, B_2)$ constructed in
Theorem \ref{thm:chd} with $\mu=1$ has in our case the form
\eqref{compl_sym_1}--\eqref{compl_sym_2}, and
\eqref{z_1}--\eqref{z_2} become \eqref{z_1_sym}--\eqref{z_2_sym}, this
completes the proof.
\end{proof}
\begin{rem}\label{proc_sym}
\rm{The procedure for checking whether $\csd(A_1,A_2)=1$, and if
this is the case --- for finding all minimal commuting symmetric
completions of a pair of symmetric matrices $(A_1,A_2)$, can be
obtained as the specialization of the procedure mentioned in Section
\ref{herm} (which, in turn,  is based on the procedure from Section
\ref{shorter}) by setting $u_2=\overline{u}_1$ and $\mu=1$ (see
Theorem \ref{thm:csd} and its preceding paragraph).}
\end{rem}

\section{The separability problem}\label{sec:sep}

In the 1980s the use of quantum systems as computing devices
started to being explored. The idea gained momentum when Peter
Shor \cite{S} presented a quantum algorithm for factoring a large
composite integer $N$ that was polynomial in the number of digits in
$N$. An excellent overview article on the subject of quantum
computing is \cite{Bennett_Nature}.

The separability problem occurs when a quantum system is divided
into parts. For convenience we consider a bipartite system. The state
of the system is described by a density matrix $M$, a positive
semidefinite matrix with trace 1. A state is called {\em separable}
when it can be written as a convex combination of so-called pure
separable states, i.e., $ \rho = \sum_{i=1}^k p_i \; \psi_i \psi_i^*
\otimes \phi_i \phi_i^* $ where $\psi_i$ and $\phi_i $ are (nonzero)
state vectors in the spaces corresponding to two parts of the system,
and $p_i
> 0$. When $\psi_i\in {\mathbb C}^m$ and
$\phi_i\in{\mathbb C}^n$, the matrix $\rho$ is called $m\times n$
separable. The number $k$ is referred to as the number of states
in the representation.

The problem whether a given state is separable or entangled (= not
separable) may be stated as a semi-algebraic one, and is therefore
decidable by the Tarski-Seidenberg decision procedure \cite{BCR}. As
it turns out though, the separability problem scales very poorly with
the number of variables and these techniques are in general not
practical. In fact, the separability problem in its full generality has been
shown to be NP-complete \cite{Gu}.

As a consequence of the results of Section~\ref{sec:compl} we can
state a new result for the $2\times n$ case. Thus we are concerned
with matrices
\begin{equation}
\label{pd2} M=\begin{bmatrix} A & B^* \\ B & C \end{bmatrix} \ge
0.
\end{equation}
Notice that if $M=\sum_{i=1}^k P_i \otimes Q_i$ with $P_i \in
{\mathbb C}^{2\times 2}$ positive semidefinite and $Q_i \in
{\mathbb C}^{n\times n}$ positive semidefinite, then
$\widetilde{M} = \sum_{i=1}^k P_i^T \otimes Q_i$ is positive
semidefinite as well. One easily sees that
\begin{equation}
\label{peres2a} \widetilde{M}=\begin{bmatrix} A & B \\ B^* & C
\end{bmatrix} \ge 0.
\end{equation}
Thus for \eqref{pd2} to have a chance to be $2\times n$ separable we
need \eqref{peres2a} to hold. This is referred to as the ``Peres test'';
see \cite{P}. As was observed by several authors, the $2\times n$
separability problem for \eqref{pd2} can easily be reduced to the case
when $A=I$; see, for instance, Proposition 3.1 in \cite{Wsep}. Using
Theorem 3.2 in \cite{Wsep}, which connects the separability problem
to the normal completion problem, we can now state a method for
checking separability of \eqref{pd2} in the case when $\rank(M) =
\rank (\widetilde{M}) = \rank(A) + 1$.

\begin{thm}\label{sep} Let $B$, $C \in {\mathbb C}^{n\times n}$ be such that
\begin{equation} M=\begin{bmatrix} I_n & B^* \\ B & C
\end{bmatrix} \ge 0,\quad  \widetilde{M}=\begin{bmatrix} I_n & B \\ B^* & C
\end{bmatrix} \ge 0, \end{equation} and suppose that
\begin{equation} \rank(M) = \rank(\widetilde{M}) = n +1.
\end{equation} Write
$$ C-BB^* = xx^*,\quad C-B^*B = yy^*$$ for some vectors $x,y\in{\mathbb C}^n$. Then $M$ is $2\times n$
separable if and only if $ x,y,B^*x,By $ are linearly dependent.
In this case, the minimal number of states in a separable
representation of $M$ is $n+1$.
\end{thm}

\begin{proof}
First notice that $B^*B-BB^*=xx^*-yy^*$.

Suppose that $ x,y,B^*x,By $ are linearly dependent. Then by
Theorem \ref{normalext1} there exists a normal matrix
$$N=\begin{bmatrix} B & \nu x \\ y^* & z \end{bmatrix},$$
where $|\nu|= 1$. But as $(\nu x )(\nu x)^* = C-BB^*$ it follows
from Theorem 3.2 in \cite{Wsep} that $M$ is $2\times n$ separable,
and that the minimal number of states in a separable
representation of $M$ is $n+1$.

Conversely, suppose that $M$ is $2\times n$ separable. By Theorem
3.2 in \cite{Wsep} there exists a normal matrix
$$ N = \begin{bmatrix} B & S \\ T & P \end{bmatrix} $$ so that
$BB^* + SS^* \le C$. But then $SS^* \le xx^*$ and thus $S=xv^*$
with $\| v \| \le 1$. Also $B^*B + T^* T = BB^* + SS^* \le C$, and
thus $T^*T \le yy^*$ yielding $T=yw^*$ with $\|w \|\le 1$. In
addition, $BT^* + SP^* = B^* S + T^* P$. In particular,
$\range(BT^* + SP^*) = \range(B^*S+T^*P)$. Note that $\range (BT^*
+ SP^*) \subseteq \spn (By, x)$ and $\range(B^*S+T^*P) \subseteq
\spn(B^* x, y)$. But then it follows easily that $ x,y,B^*x,By $
are linearly dependent. Indeed, if $BT^* + SP^* = B^* S + T^* P
\neq 0$, then $\spn(By, x)$ and $\spn(B^* x, y)$ must have a
nontrivial intersection, and if $BT^* + SP^* = B^* S + T^* P = 0$,
then $\spn(By, x)$ and $\spn(B^* x, y)$ are both at most one
dimensional.
\end{proof}

We can now provide a new proof of the following result by
Woronowicz \cite{Woron}.

\begin{thm}\label{Woron}  Let $A,B,C$ be $n\times n$ matrices with $n \le 3$, so that
\begin{equation}
\label{pd2a} M=\begin{bmatrix} A & B^* \\ B & C \end{bmatrix} \ge
0,\quad \widetilde{M}=\begin{bmatrix} A & B \\ B^* & C
\end{bmatrix} \ge 0.
\end{equation}
Then $M$ is $2\times n$ separable.
\end{thm}

We will use a result by Hildebrandt which we quote without proof.

\begin{lem}\label{Hil} \cite[Lemma 2.6]{Hil} Let $K$ be a convex
cone in a real vector space ${\mathcal H}$ of finite dimension
$N$, and let ${\mathcal L} \subseteq {\mathcal H}$ be a subspace
of dimension $n$. Let $K' = K\cap {\mathcal L}$ and $x$ the
generator of an extreme ray in $K'$. Then the minimal face in $K$
containing $x$ has dimension at most $N-n+1$.
\end{lem}

Notice that if we consider the cone $PSD_n$ of  $n\times n$
complex positive semidefinite matrices, then the minimal face
containing $M \ge 0$, is the cone $F=\{ GCG^*\colon   C \in PSD_k
\}$, where $M=GG^*$ with $\Null(G) =\{ 0\}$, and $k = \rank(M)$.
In particular, the real dimension of this minimal face is
$(\rank(M))^2$.

\begin{proof}[Proof of Theorem \ref{Woron}] Since the case $n<3$
can be embedded into the case $n=3$, we will focus on the latter.
As the $2\times n$ separable matrices form a convex cone, it
suffices to prove the result for pairs $(M,\widetilde{M})$ that
generate extreme rays in the cone of pairs of matrices as in
\eqref{pd2a}. If we apply Lemma \ref{Hil} with the choices of $K =
PSD_6 \times PSD_6$ and ${\mathcal L}$ the subspace
$$\left\{ \left( \begin{bmatrix} A & B^* \\ B & C \end{bmatrix}, \begin{bmatrix} A & B \\ B^* & C
\end{bmatrix}\right)\right\}$$
in the (real) vector space of pairs of Hermitian matrices of size
$6\times 6$, then $K'$ is the cone of pairs of matrices as in
\eqref{pd2a}. By Lemma \ref{Hil} the minimal faces in $K$ containing
extreme rays of $K'$ cannot have dimension greater than
$72-36+1=37$. However, the minimal face in $K$ containing
$(M,\widetilde{M})$ (which generates an extreme ray in $K'$) has
dimension $(\rank(M))^2 + (\rank(\widetilde{M}) )^2$, and hence the
vector $(\rank(M), \rank(\widetilde{M}) )\in\mathbb{R}^2$ lies in the
closed disk of radius $\sqrt{37}$ centered at the origin. This now gives
that either $\min \{ \rank(M),\rank(\widetilde{M}) \} \le 3$ or $\max \{
\rank(M),\rank(\widetilde{M}) \} = 4. $ Next, as in Proposition 3.1 in
\cite{Wsep} we can assume that $A=I$. If now $\min \{ \rank (M),\rank
(\widetilde{M}) \} \le 3$ we have that $C=BB^*=B^*B$, and thus $B$ is
normal, which yields by Theorem 3.2 in \cite{Wsep} that $M$ is
$2\times n$ separable. On the other hand, if $\max \{
\rank(M),\rank(\widetilde{M}) \} = 4$ we can conclude by Theorem
\ref{sep} that $M$ is $2\times n$ separable (as 4 vectors in ${\mathbb
C}^n$ are always linearly dependent when $n\le 3$).
\end{proof}

It should be noted that the original statement of Woronowicz is
formulated in the dual form: if $\Phi : {\mathbb C}^{2\times 2}
\to {\mathbb C}^{n\times n}$ is a positive linear map (thus $\Phi
( PSD_2 ) \subseteq PSD_n$) and $n \le 3$, then $\Phi$ must be
decomposable. That is, $\Phi$ must be of the form $\Phi (M) =
\sum_{i=1}^k R_i M R_i^* + \sum_{i=1}^l S_i M^T S_i^*$. \medskip

{\it Acknowledgment.} The authors wish to thank Dr. Roland
Hildebrand for pointing out how Theorem \ref{sep} in conjunction
with his results, leads to a proof of the result by Woronowicz.

\end{document}